\documentclass[reqno]{amsart}

\usepackage{graphicx}
\usepackage{hyperref}
\usepackage{ellipsis}
\usepackage{amsmath,amsthm, amssymb}
\usepackage{enumerate}
\usepackage{comment}
\usepackage[font=small,labelfont=bf]{caption}
\usepackage{subcaption}
\usepackage{url}
\usepackage{cite}

\theoremstyle{plain}
\newtheorem{theorem}[equation]{Theorem}

\newtheorem{lemma}[equation]{Lemma}

\theoremstyle{definition}

\renewcommand{\vec}[1]{{\bf #1}}

\def\Z{{\mathbb{Z}}}
\def\T{{\mathbb{T}}}
\def\R{{\mathbb{R}}}

\newcommand{\C}{\mathbb{C}}
\newcommand{\h}{\mathbb{H}}
\def\C{{\mathbb{C}}}

\newcommand{\leg}[2]{\left(\frac{#1}{#2}\right)}

\newcommand{\e}[1]{e\left(#1\right)}

\newcommand{\ep}[1]{e\left(\frac{#1}{p}\right)}
\newcommand{\epsmall}[1]{e\Big(\frac{#1}{p}\Big)}
\newcommand{\eqsmall}[1]{e\Big(\frac{#1}{q}\Big)}
\renewcommand{\Re}{\operatorname{Re}}
\renewcommand{\Im}{\operatorname{Im}}

\title[Generalized Kloosterman sums]{Visual properties of\\generalized Kloosterman sums}
    
\author[P.~Burkhardt]{Paula Burkhardt}
    
\author[A.Z.-Y.~Chan]{Alice Zhuo-Yu Chan}
    \address{   Department of Mathematics\\
            University of California, San Diego\\
            La Jolla, California\\
            92093 \\ USA}
    \email{azchan@ucsd.edu}
    
\author[G.~Currier]{Gabriel Currier}
    
\author[S.R.~Garcia]{Stephan Ramon Garcia}
    \address{   Department of Mathematics\\
            Pomona College\\ 610 N. College Ave\\
            Claremont, California\\
            91711 \\ USA}
    \email{Stephan.Garcia@pomona.edu}
    \urladdr{\url{http://pages.pomona.edu/~sg064747}}
    
\author[F.~Luca]{Florian Luca}    
\address{School of Mathematics\\University of the Witwatersrand\\Private Bag 3, Wits 2050, Johannesburg, South Africa}
\email{Florian.Luca@wits.ac.za}

\author[H.~Suh]{Hong Suh}
    

\thanks{Partially supported by NSF Grant DMS-1001614 and a NSF graduate research fellowship.}

\begin{document}

\maketitle

\begin{abstract}
    For a positive integer $m$ and a subgroup $\Lambda$ of the unit group $(\Z/m\Z)^\times$,
    the corresponding \emph{generalized Kloosterman sum} is the function
    $K(a,b,m,\Lambda) = \sum_{u \in \Lambda}e(\frac{au + bu^{-1}}{m})$.
    Unlike classical Kloosterman sums, which are real valued, generalized Kloosterman sums 
    display a surprising array of visual features when their values 
    are plotted in the complex plane.  In a variety of instances,
    we identify the precise number-theoretic conditions that give rise to particular phenomena.
\end{abstract}

\section{Introduction}
    For a positive integer $m$ and a subgroup $\Lambda$ of the unit group $(\Z/m\Z)^\times$,
    the corresponding \emph{generalized Kloosterman sum} is the function
    \begin{equation}\label{eq:Kabmg}
        K(a,b,m,\Lambda) = \sum_{u \in \Lambda}e\left(\frac{au + bu^{-1}}{m}\right),
    \end{equation} 
    in which $e(x)=\exp (2\pi i x)$ and $u^{-1}$ denotes the multiplicative inverse of $u$ modulo $m$.
    The classical Kloosterman sum arises when $\Lambda = (\Z/m\Z)^{\times}$ \cite{iwaniec2004analytic}.
        
    Unlike their classical counterparts, which are real valued, generalized Kloosterman sums 
    display a surprising array of visual features when their values 
    are plotted in the complex plane; see Figure \ref{fig:various1}.  Our aim here is
    to initiate the investigation of these sums from a graphical perspective.  In a variety of instances,
    we identify the precise number-theoretic conditions that give rise to particular phenomena.

    \begin{figure}
    	\centering
    	\begin{subfigure}[b]{0.30\textwidth}
    		\includegraphics[width=\textwidth]{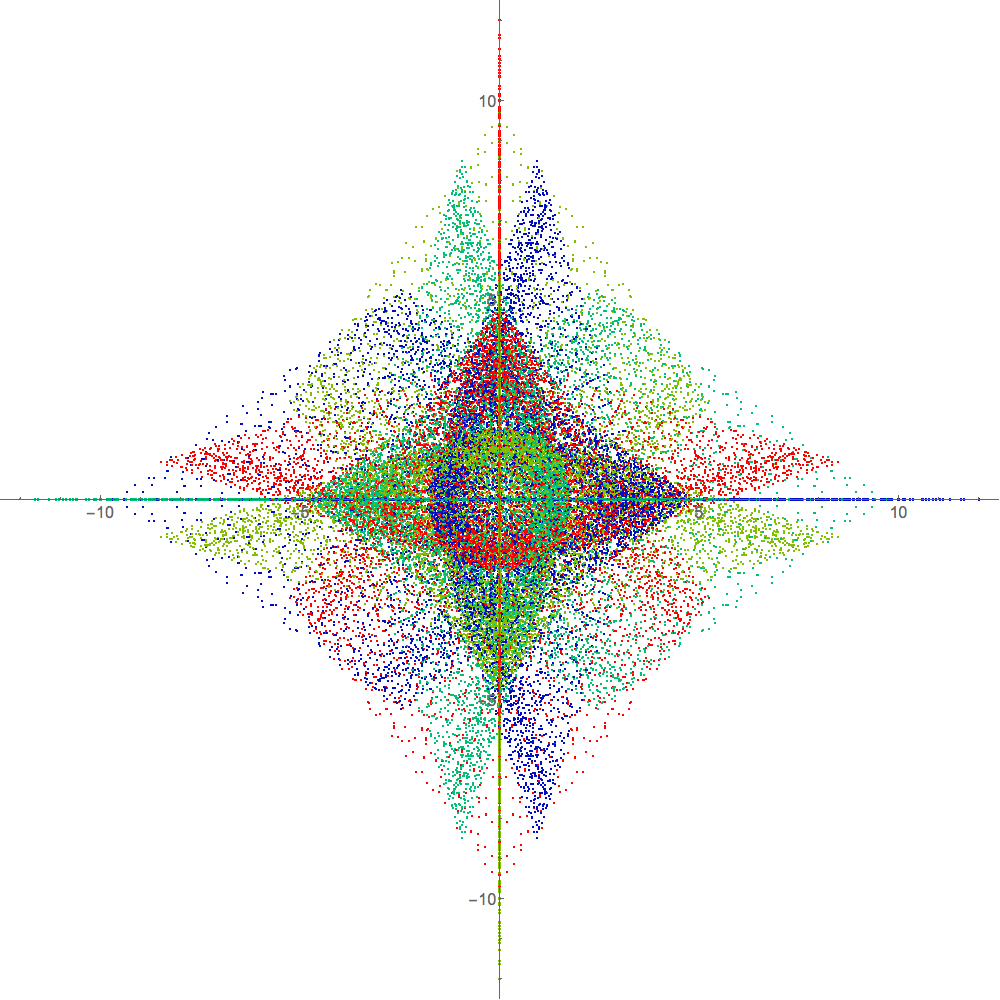}
    		\caption{$m=4820$, $\Lambda=\langle 1209 \rangle$}
    	\end{subfigure}
    	\quad
    	\begin{subfigure}[b]{0.30\textwidth}
    		\includegraphics[width=\textwidth]{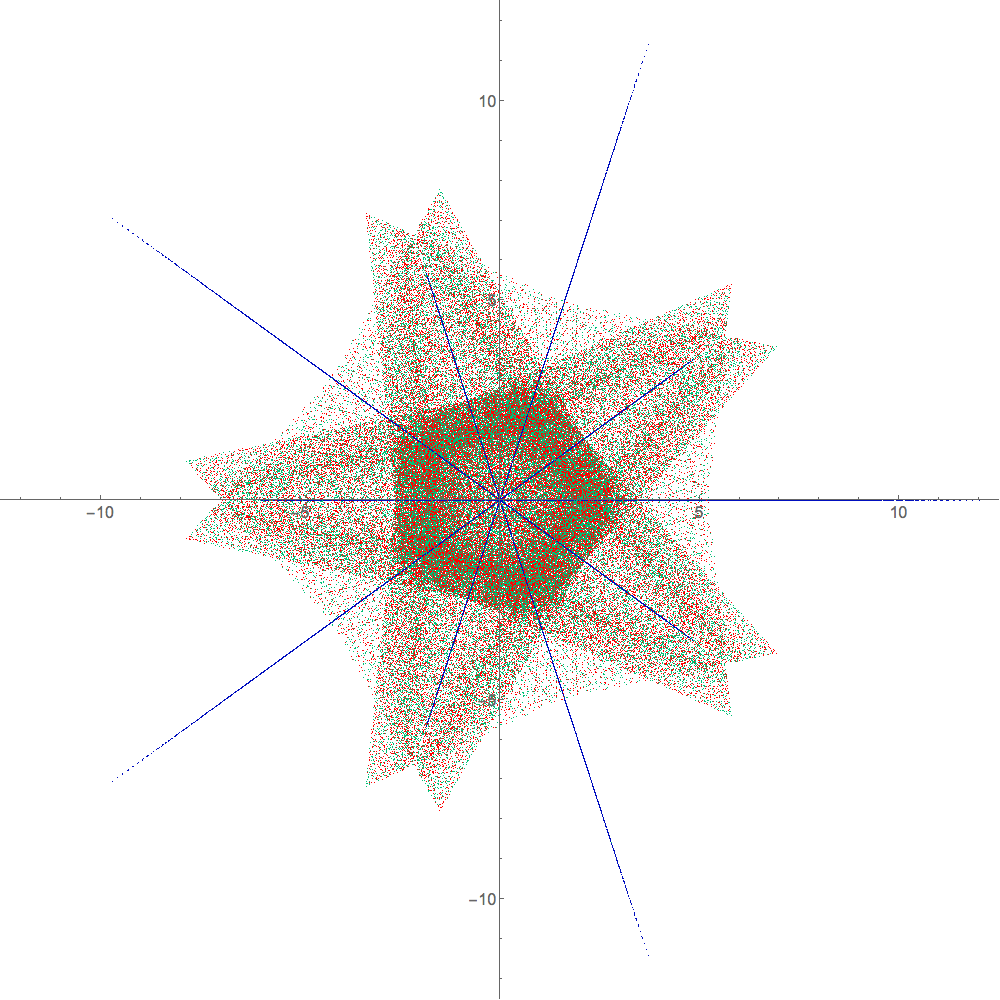}
    		\caption{$m=9015$, $\Lambda=\langle 596 \rangle$}
    	\end{subfigure}
    	\quad
    	\begin{subfigure}[b]{0.30\textwidth}
    		\includegraphics[width=\textwidth]{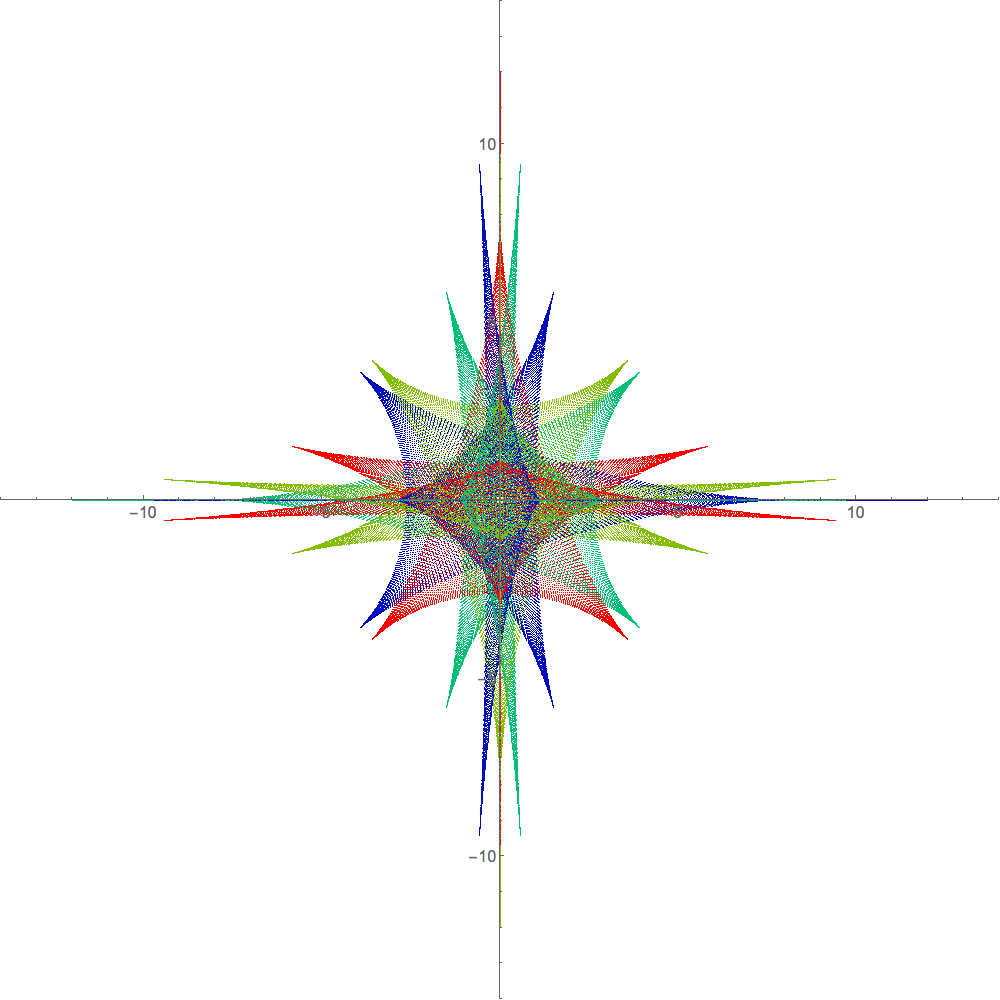}
    		\caption{$m=4820$, $\Lambda = \langle 257 \rangle$}
    	\end{subfigure}
    	\quad
    
    	\begin{subfigure}[b]{0.30\textwidth}
    		\includegraphics[width=\textwidth]{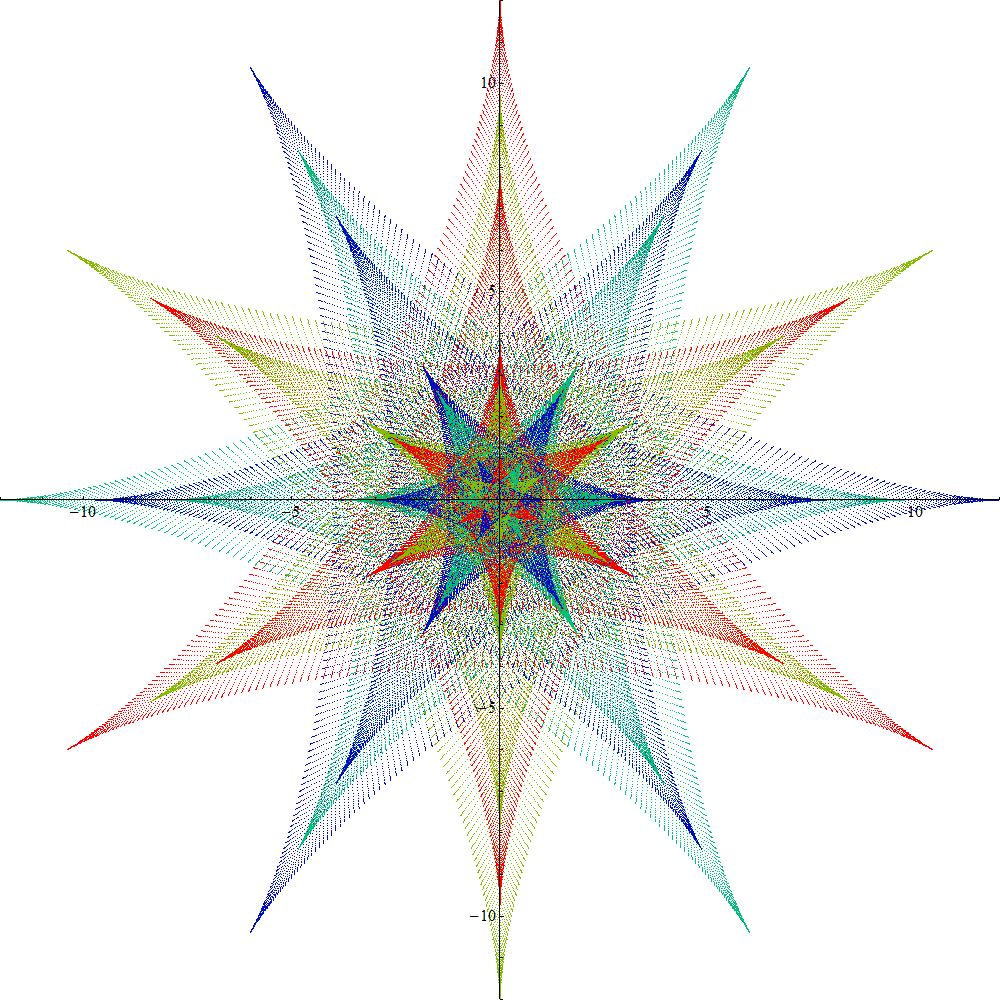}
    		\caption{$m=4820$, $\Lambda = \langle 497 \rangle$}
    	\end{subfigure}
    	\quad
    	\begin{subfigure}[b]{0.30\textwidth}
    		\includegraphics[width=\textwidth]{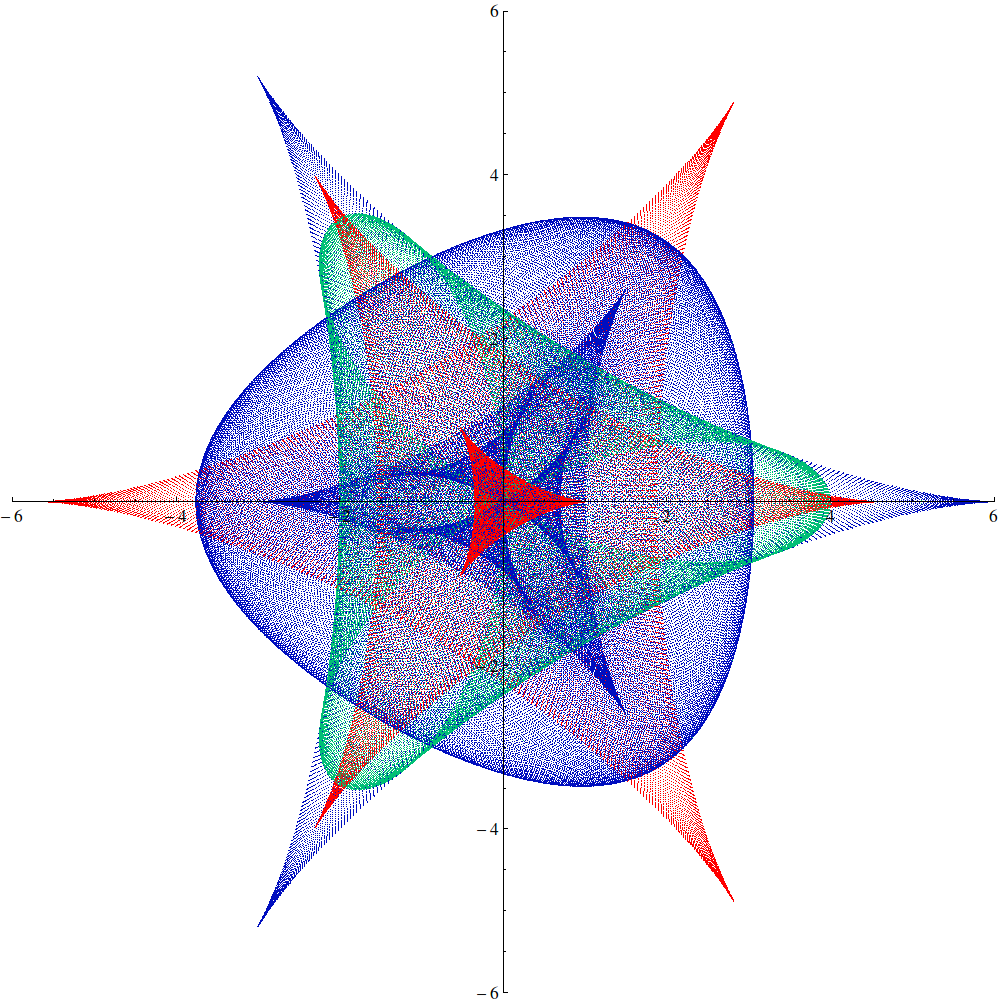}
    		\caption{$m=3087$, $\Lambda = \langle 1010 \rangle$}
    	\end{subfigure}
    	\quad
    	\begin{subfigure}[b]{0.30\textwidth}
    		\includegraphics[width=\textwidth]{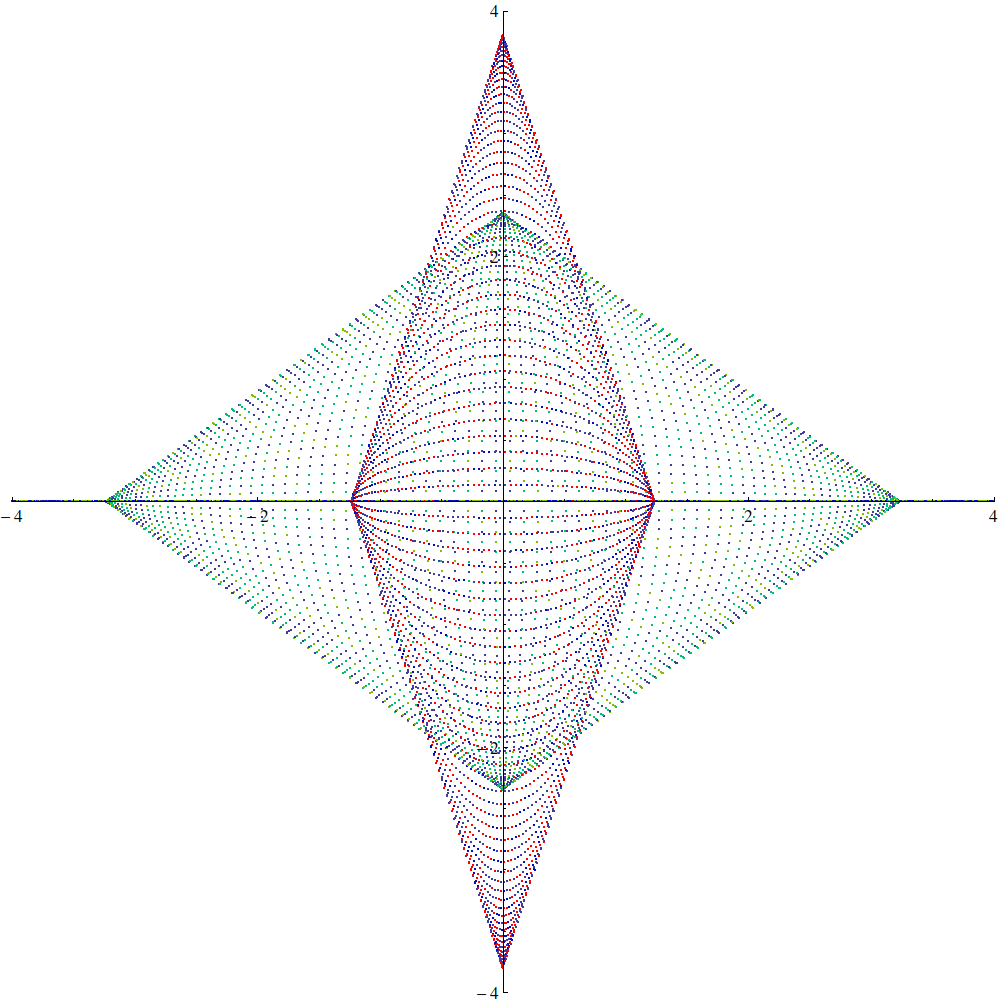}
    		\caption{$m=890$, $\Lambda = \langle 479 \rangle$}
    	\end{subfigure}	
    
    	\begin{subfigure}[b]{0.30\textwidth}
    		\includegraphics[width=\textwidth]{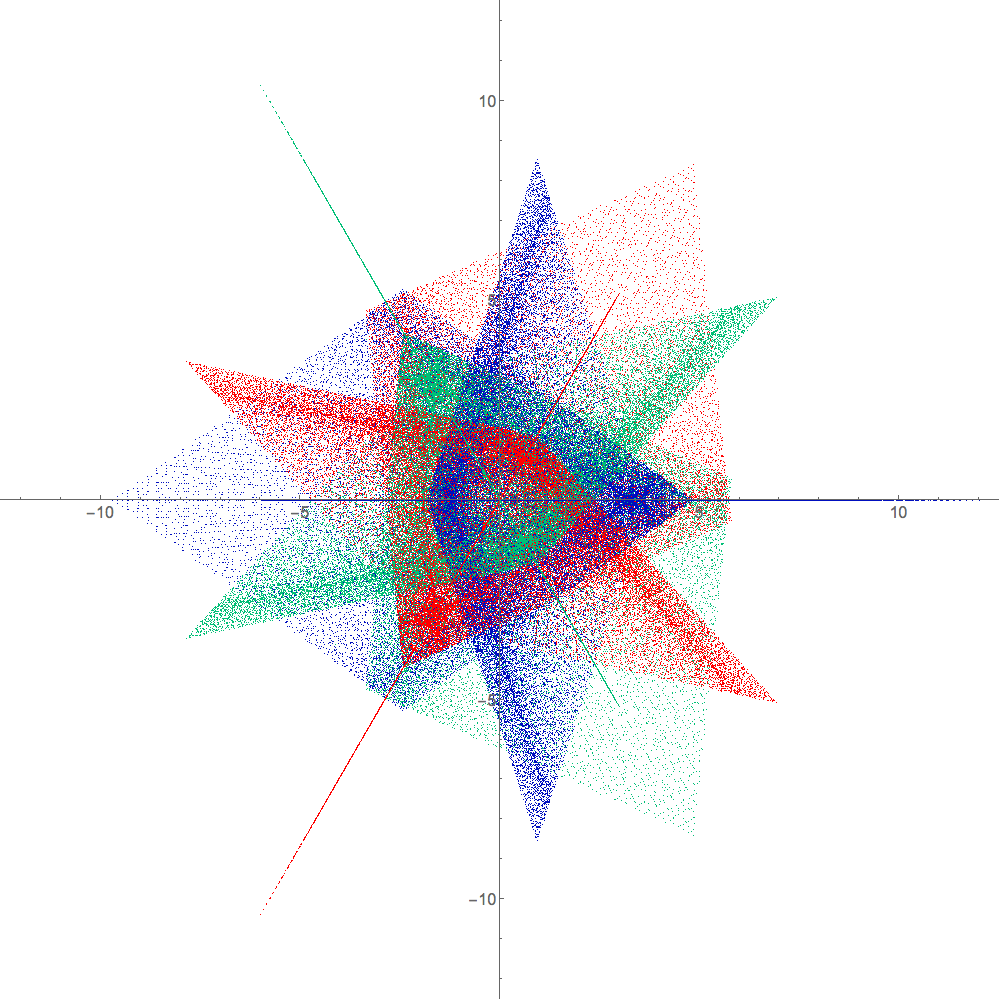}
    		\caption{$m=9015$, $\Lambda=\langle 2284 \rangle$}
    	\end{subfigure}
    	\quad
    	\begin{subfigure}[b]{0.30\textwidth}
    		\includegraphics[width=\textwidth]{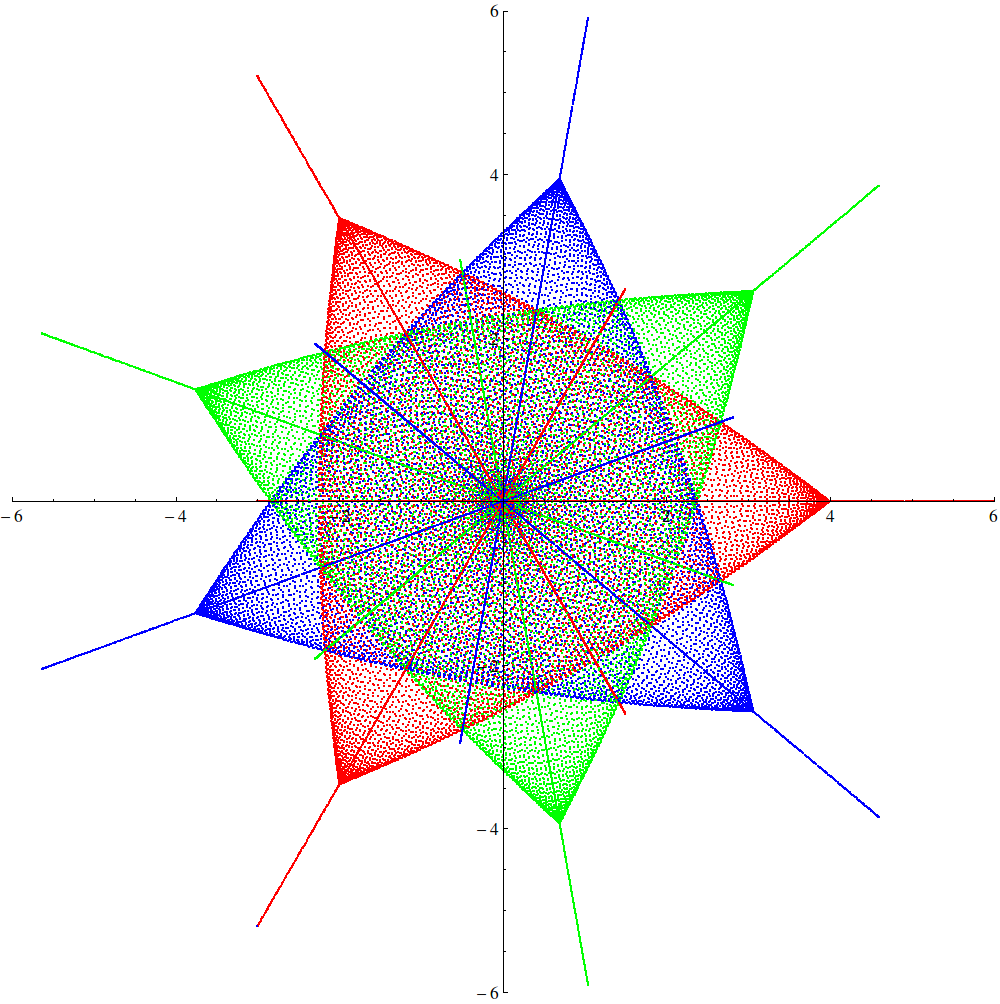}
    		\caption{$m=1413$, $\Lambda=\langle 13 \rangle$}
    	\end{subfigure}
    	\quad
    	\begin{subfigure}[b]{0.30\textwidth}
    		\includegraphics[width=\textwidth]{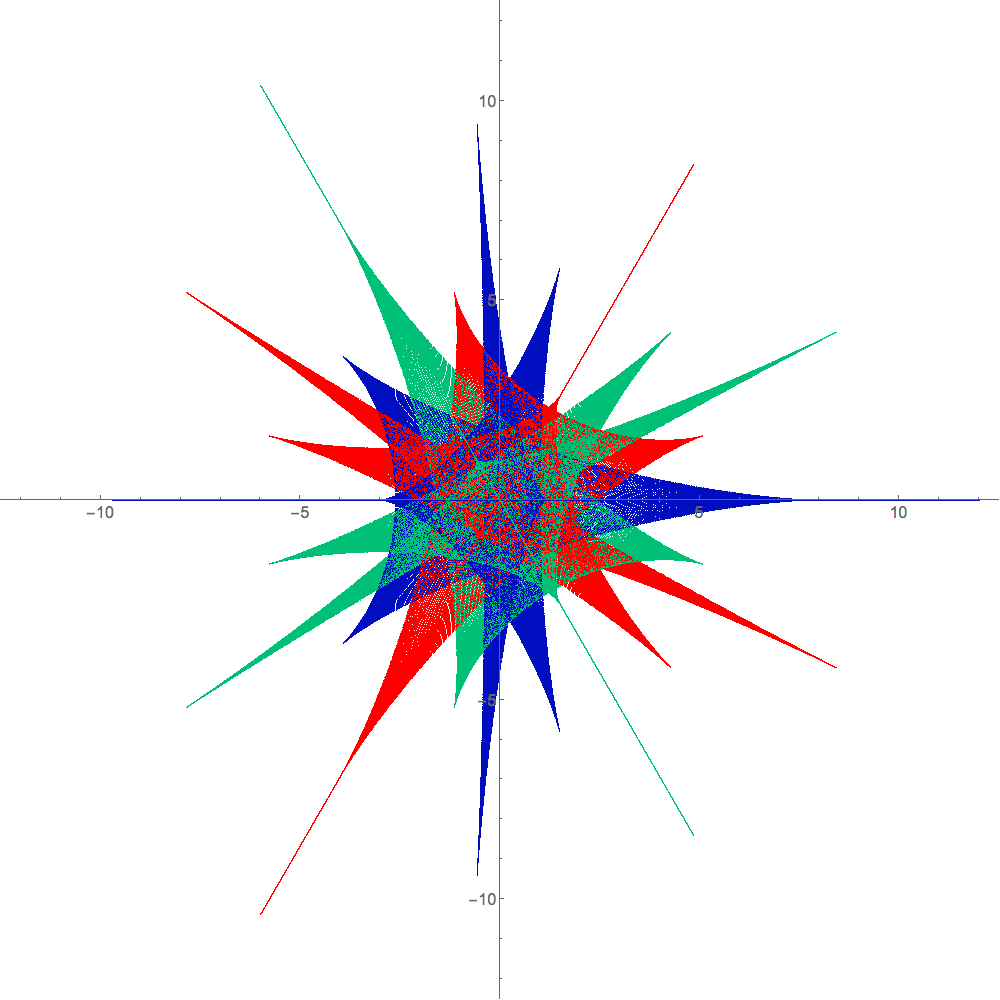}
    		\caption{$m=9015$, $\Lambda=\langle 577 \rangle$}
    	\end{subfigure}
    
    	\caption{Plots in $\C$ of the values of generalized Kloosterman sums $K(a,b,m,\Lambda)$ for $0 \leq a,b < m$.
	These plots are symmetric across the real line since $\overline{K(a,b,p,d)}=K(-a,-b,p,d)$.
    	The images are colored according to the value of $a+b \pmod{k}$ for a carefully selected $k$.}
    	\label{fig:various1}
    \end{figure}
    
    Like classical Kloosterman sums, generalized Kloosterman sums enjoy a certain multiplicative property. 
        If $m = m_1m_2$, in which $(m_1,m_2)=1$, 
        $r_1 \equiv m_1^{-1} \pmod{m_2}$, $r_2 \equiv m_2^{-1} \pmod{m_1}$,
        $\omega_1 = \omega \pmod {m_1}$, and $\omega_2 = \omega \pmod {m_2}$, then
        \begin{equation}\label{eq:Multiplicative}
            K\big(a,b,m,\langle \omega \rangle \big) = K\big(r_2a, r_2b, m_1, \langle \omega_1 \rangle\big)K\big(r_1a, r_1b, m_2, \langle  \omega_2 \rangle\big).
        \end{equation}        
        This follows immediately from the Chinese Remainder Theorem.    
    Consequently, we tend to focus on prime or prime power moduli;
    see Figure \ref{fig:decomp}.  Since the group of units modulo an odd prime power is cyclic, most of our attention is restricted to the case
    where $\Gamma = \langle \omega \rangle$ is a cyclic group of units.
    
    \begin{figure}
    	\centering
    	\begin{subfigure}[b]{0.30\textwidth}
    		\includegraphics[width=\textwidth]{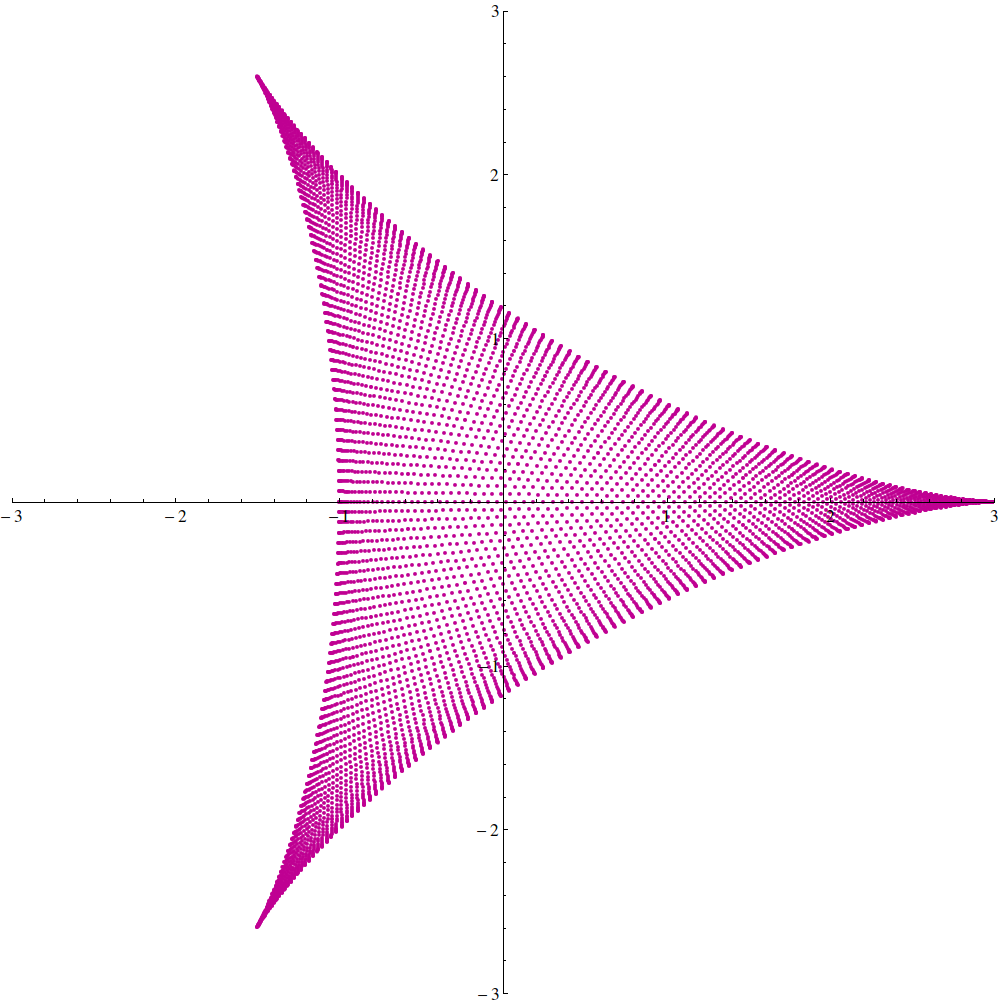}
    		\caption{$p=199$, $d=3$}
		\label{fig:decomp:a}
    	\end{subfigure}
    	\quad
    	\begin{subfigure}[b]{0.30\textwidth}
    		\includegraphics[width=\textwidth]{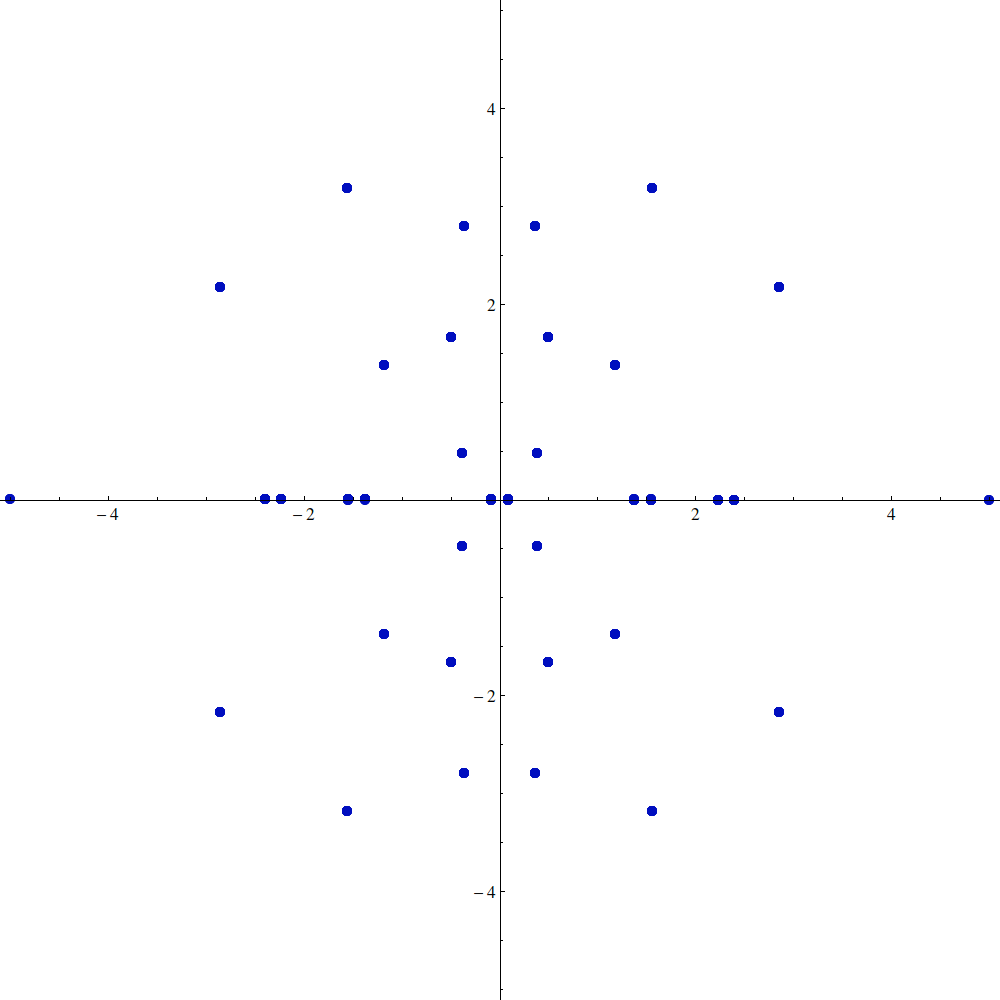}
    		\caption{$m=22$, $\Lambda=\langle 5 \rangle$}
    	\end{subfigure}
    	\quad
    	\begin{subfigure}[b]{0.30\textwidth}
    		\includegraphics[width=\textwidth]{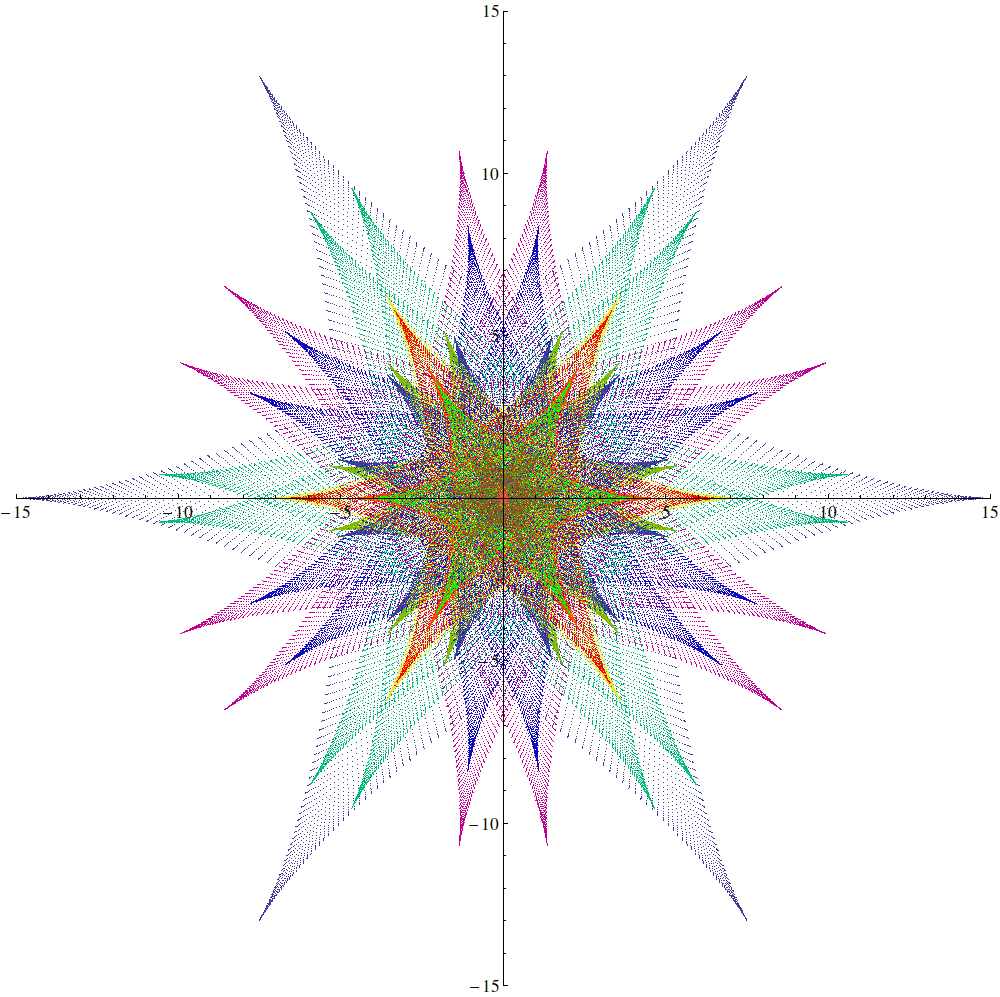}
    		\caption{$m=4378$, $\Lambda=\langle 291 \rangle$}	
    		\label{fig:decomp:c}	
    	\end{subfigure}
    
    	\caption{The values $K(a,b,4378,\langle 291 \rangle)$ are obtained from the values of
    	$K(\,\cdot\,,\,\cdot\,,199,\langle 92 \rangle)$ and $K(\,\cdot\,,\,\cdot\,,22,\langle 5 \rangle)$
    	using \eqref{eq:Multiplicative}.  The image in Figure \ref{fig:decomp:a}
	is explained by Theorem \ref{thm:hypocycloid}.
    	The points $K(a,b,4378,\langle 291 \rangle)$ in Figure \ref{fig:decomp:c} are colored according to the value of $a+b \pmod{22}$. }
    	\label{fig:decomp}
    \end{figure}

    Additional motivation for our work stems from the fact that generalized Kloosterman sums
    are examples of supercharacters.
    The theory of supercharacters, introduced in 2008 by P.~Diaconis and I.M.~Isaacs \cite{diaconis2008supercharacters}, has emerged as
    a powerful tool in combinatorial representation theory.  
    Certain exponential sums of interest in number theory, such as 
    Ramanujan, Gauss, Heilbronn, and classical Kloosterman sums, arise as supercharacter values on abelian groups
    \cite{brumbaugh2012supercharacters,duke,gausscyclotomy, fowler, heilbronn}. 
    In the terminology of \cite{brumbaugh2012supercharacters}, the functions \eqref{eq:Kabmg} arise
    by letting $n =m$, $d= 2$, and $\Gamma = \{ \operatorname{diag}(u,u^{-1}) : u \in \Lambda\}$.

\section{Hypocycloids}

    In what follows, we let $\phi$ denote the Euler totient function.
    If $q = p^{\alpha}$ is an odd prime power, then $(\Z/q\Z)^{\times}$ is cyclic. Thus, for each divisor $d$ of $\phi(q) = p^{\alpha-1}(p - 1)$,
    there is a unique subgroup $\Lambda$ of $(\Z/q\Z)^\times$ of order $d$.  In this case, we write
    \begin{equation*}
        K(a,b,q,d)=\sum_{u^d=1} e\left( \frac{au+bu^{-1}}{q} \right)
    \end{equation*}
    instead of $K(a,b,q,\Lambda)$.
    Under certain conditions, these generalized Kloosterman sums display remarkable asymptotic behavior.
    If $d$ is a fixed odd prime and $q=p^\alpha$ is an odd prime power with $p \equiv 1 \pmod{d}$, then
    the values $K(a,b,q,d)$ for $0 \leq a,b < q$ are contained in the closure $\h_d$ of the bounded region determined by the 
    $d$-cusped hypocycloid given by
    \begin{equation}\label{eq:Parameter}
        \theta\mapsto (d-1)e^{i\theta}+e^{(1-d)i\theta};
    \end{equation}
    see Figure \ref{Figure:Hypocycloid}.
    As the prime power $q=p^\alpha$ for $p \equiv 1 \pmod{d}$ tends to infinity, the values $K(a,b,q,d)$
    ``fill out'' $\h_d$; see Figure \ref{fig:hypocycloid}. 
    Similar asymptotic behavior has been observed in Gaussian periods
    \cite{gausscyclotomy, duke} and certain exponential sums related to the symmetric group \cite{brumbaugh2013graphic}.
    
    \begin{figure}
    	\begin{subfigure}[b]{0.25\textwidth}
    		\includegraphics[width=\textwidth]{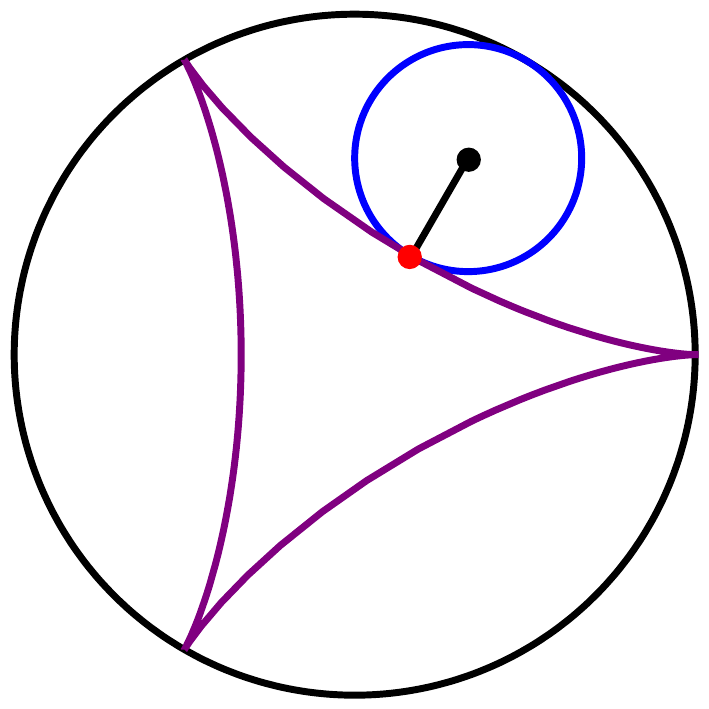}
    		\caption{$d=3$\\a deltoid}
    	\end{subfigure}
    	\qquad
    	\begin{subfigure}[b]{0.25\textwidth}
    		\includegraphics[width=\textwidth]{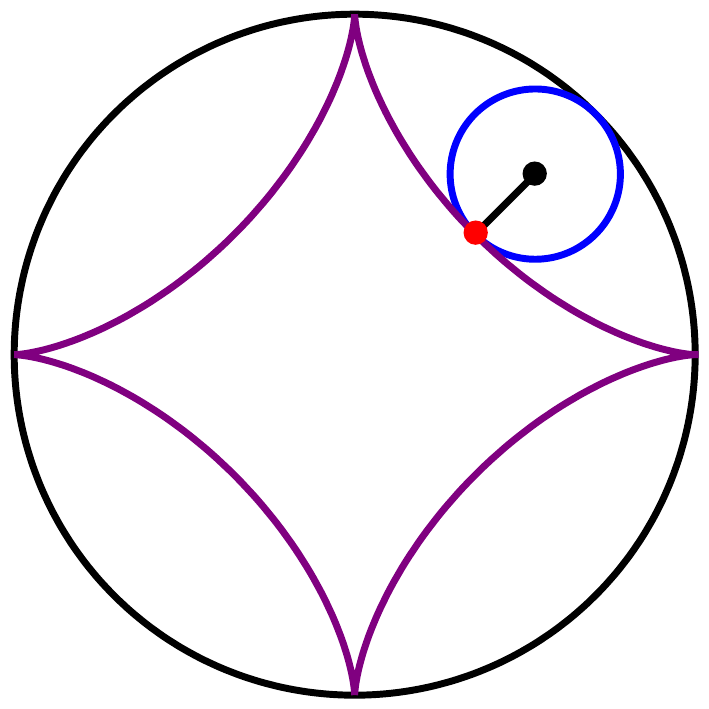}
    		\caption{$d=4$\\an astroid}
    	\end{subfigure}
    	\qquad
    	\begin{subfigure}[b]{0.25\textwidth}
    		\includegraphics[width=\textwidth]{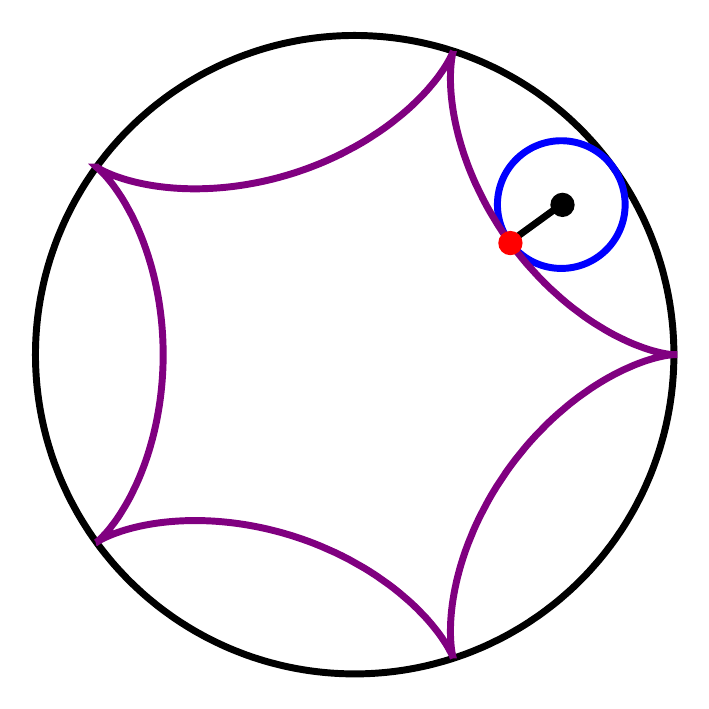}
    		\caption{$d=5$\\{\quad}}
    	\end{subfigure}
    	\caption{The curve \eqref{eq:Parameter} is obtained by rolling a circle of radius $1$ within the circle $|z| = d$; this is a hypocycloid
	with $d$ cusps.  The set $\h_d$ is the closure of the bounded region determined by the hypocycloid.}
    	\label{Figure:Hypocycloid}
    \end{figure}    
    
    To be more precise, we require a few words about uniformly distributed sets.
    A sequence $S_n$ of finite subsets of $[0,1)^k$ is \emph{uniformly distributed} if
	\begin{equation*}
		\lim_{n\to\infty}\sup_B \left| \frac{ |B\cap S_n|}{|S_n|} - \mu(B)\right| = 0,
	\end{equation*}
	where the supremum runs over all boxes $B = [a_1,b_1)\times\cdots \times [a_k,b_k)$ in $[0,1)^k$
	and $\mu$ denotes $k$-dimensional Lebesgue measure.  
	If $S_n$ is a sequence of finite subsets of $\R^k$, then
	$S_n$ is \emph{uniformly distributed mod $1$} if the sets 
	\begin{equation*}
            	\big\{  \big( \{x_1\},\{x_2\},\ldots,\{x_k\}\big) : (x_1,x_2,\ldots,x_k) \in S_n\big\}
	\end{equation*}
	are uniformly distributed in $[0,1)^k$.  Here $\{x\}$ denotes the fractional part $x-\lfloor x \rfloor$
	of a real number $x$.  The following result tells us that a certain sequence of sets closely
	related to generalized Kloosterman sums is uniformly distributed modulo $1$.

    \begin{figure}
    	\centering
    	\begin{subfigure}[b]{0.30\textwidth}
    		\includegraphics[width=\textwidth]{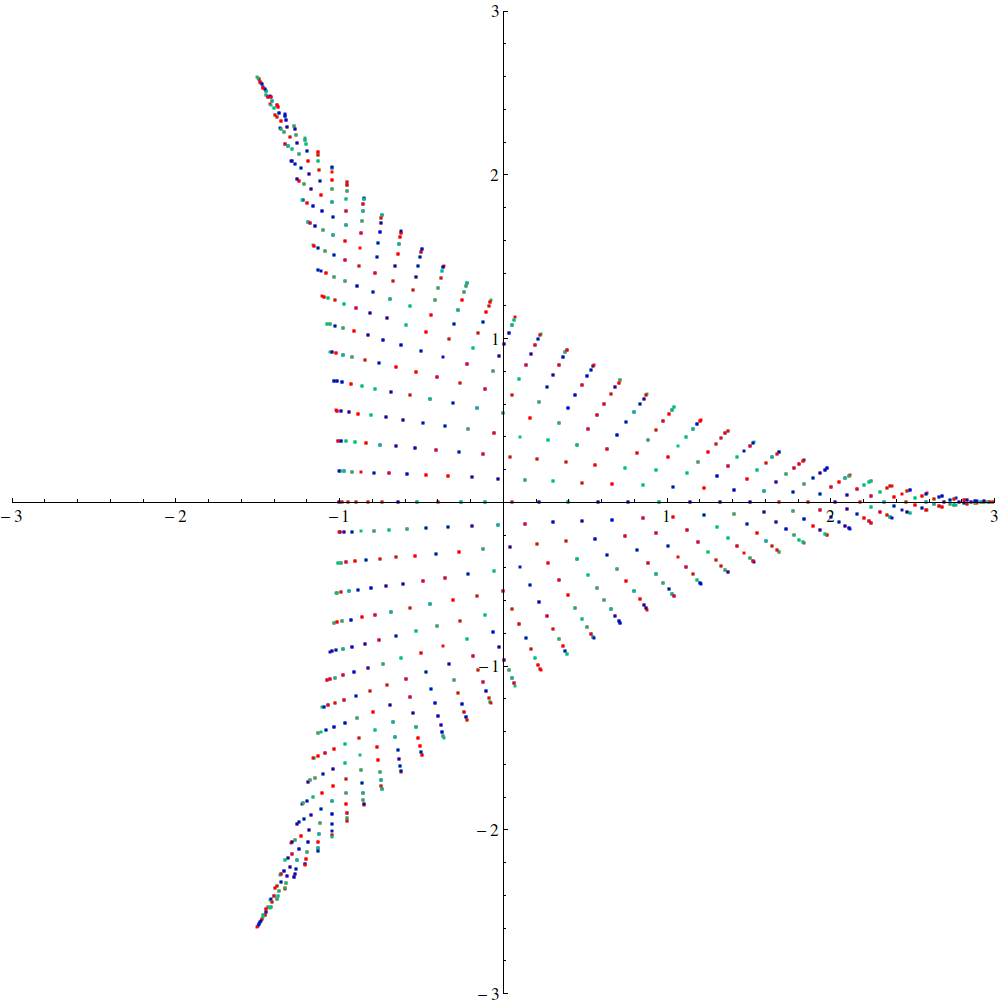}
    		\caption{$p=67$, $d=3$}
    	\end{subfigure}
	\quad
    	\begin{subfigure}[b]{0.30\textwidth}
    		\includegraphics[width=\textwidth]{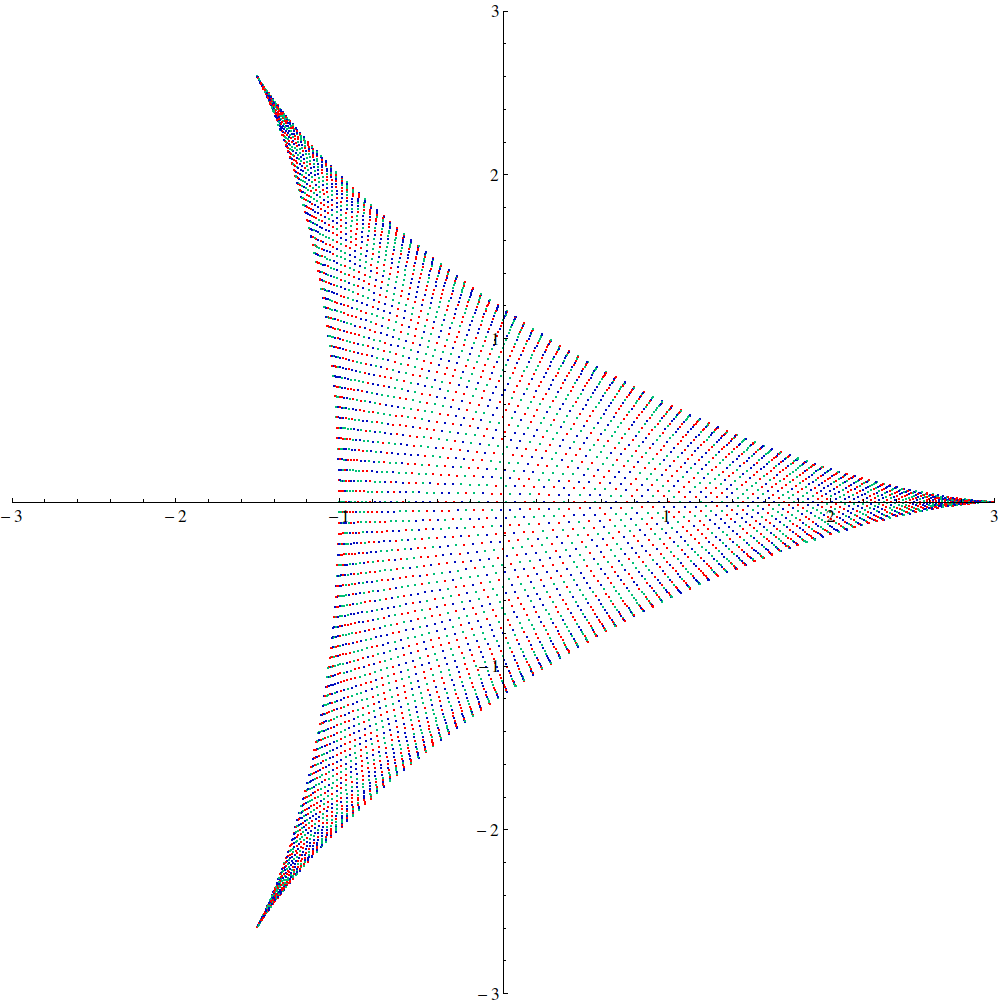}
    		\caption{$p=193$, $d=3$}
    	\end{subfigure}
	\quad
    	\begin{subfigure}[b]{0.30\textwidth}
    		\includegraphics[width=\textwidth]{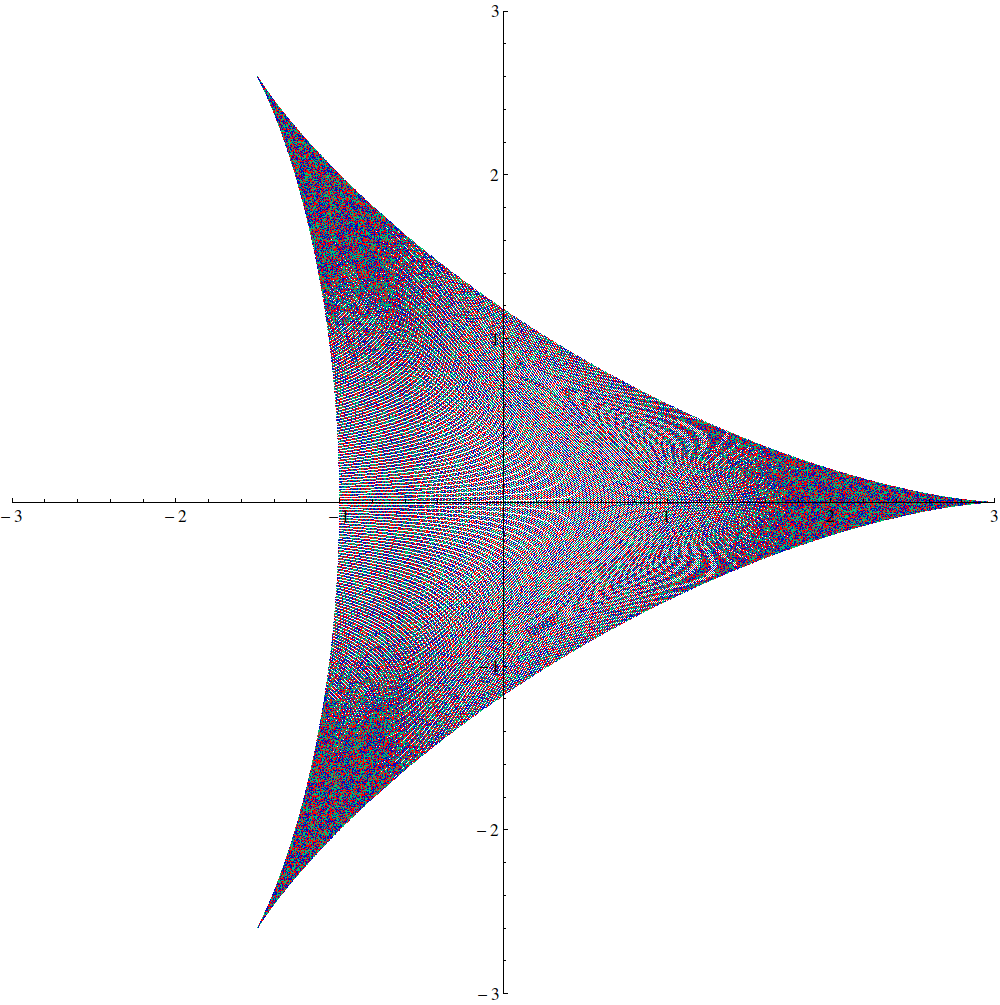}
    		\caption{$p=1279$, $d=3$}
    	\end{subfigure}
    
    	\begin{subfigure}[b]{0.30\textwidth}
    		\includegraphics[width=\textwidth]{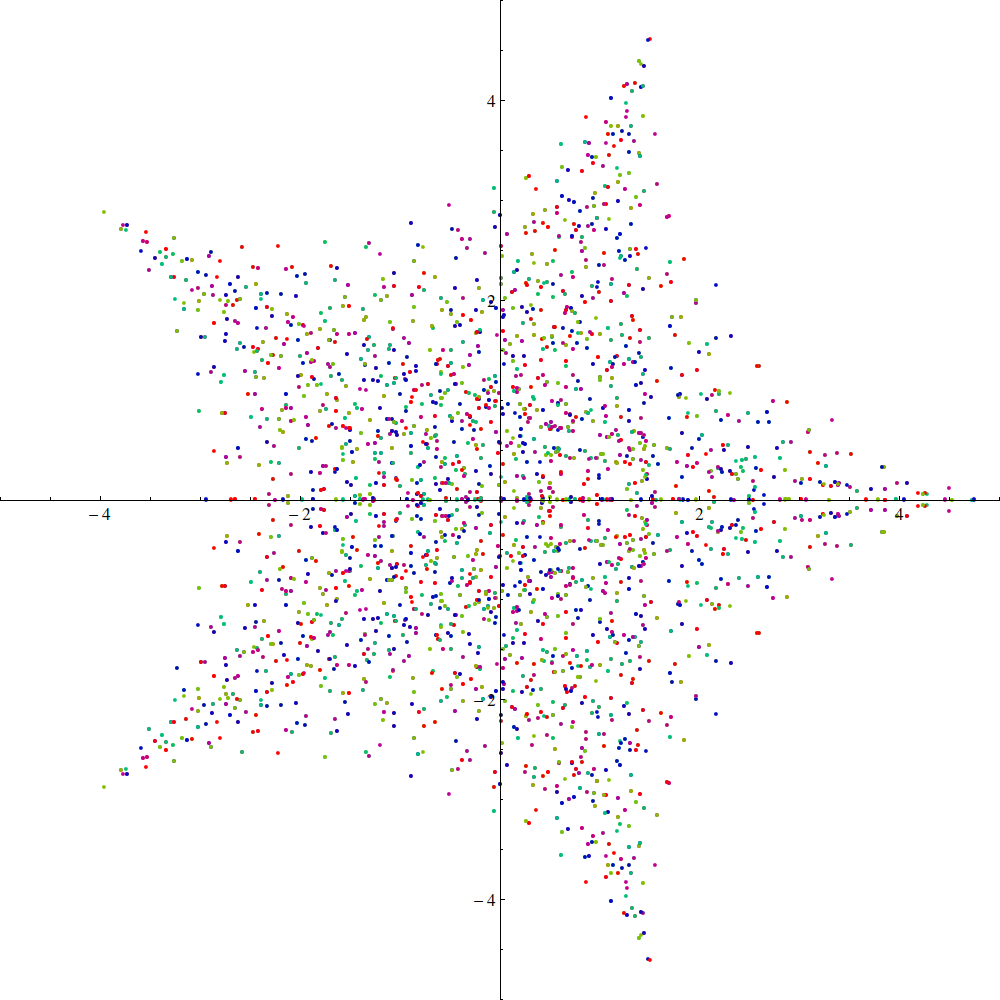}
    		\caption{$p=151$, $d=5$}
    	\end{subfigure}
	\quad
    	\begin{subfigure}[b]{0.30\textwidth}
    		\includegraphics[width=\textwidth]{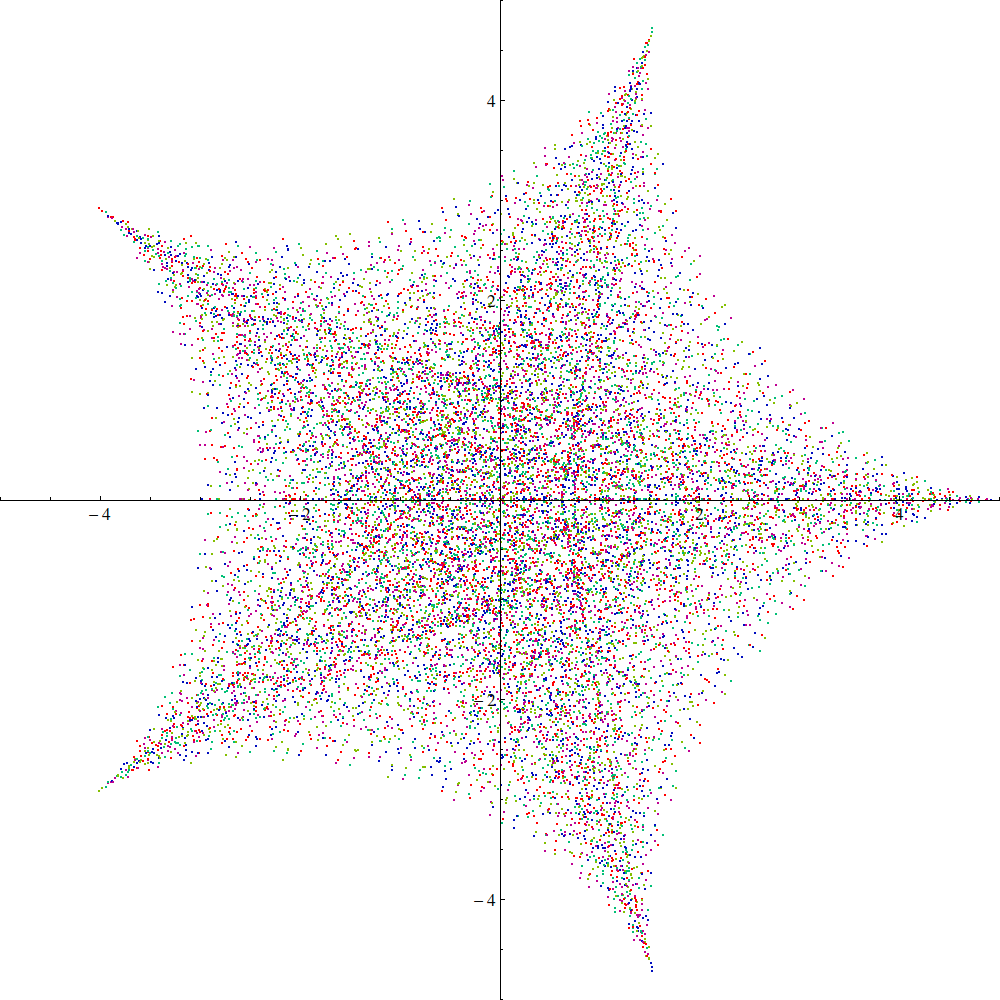}
    		\caption{$p=431$, $d=5$}
    	\end{subfigure}
	\quad
    	\begin{subfigure}[b]{0.30\textwidth}
    		\includegraphics[width=\textwidth]{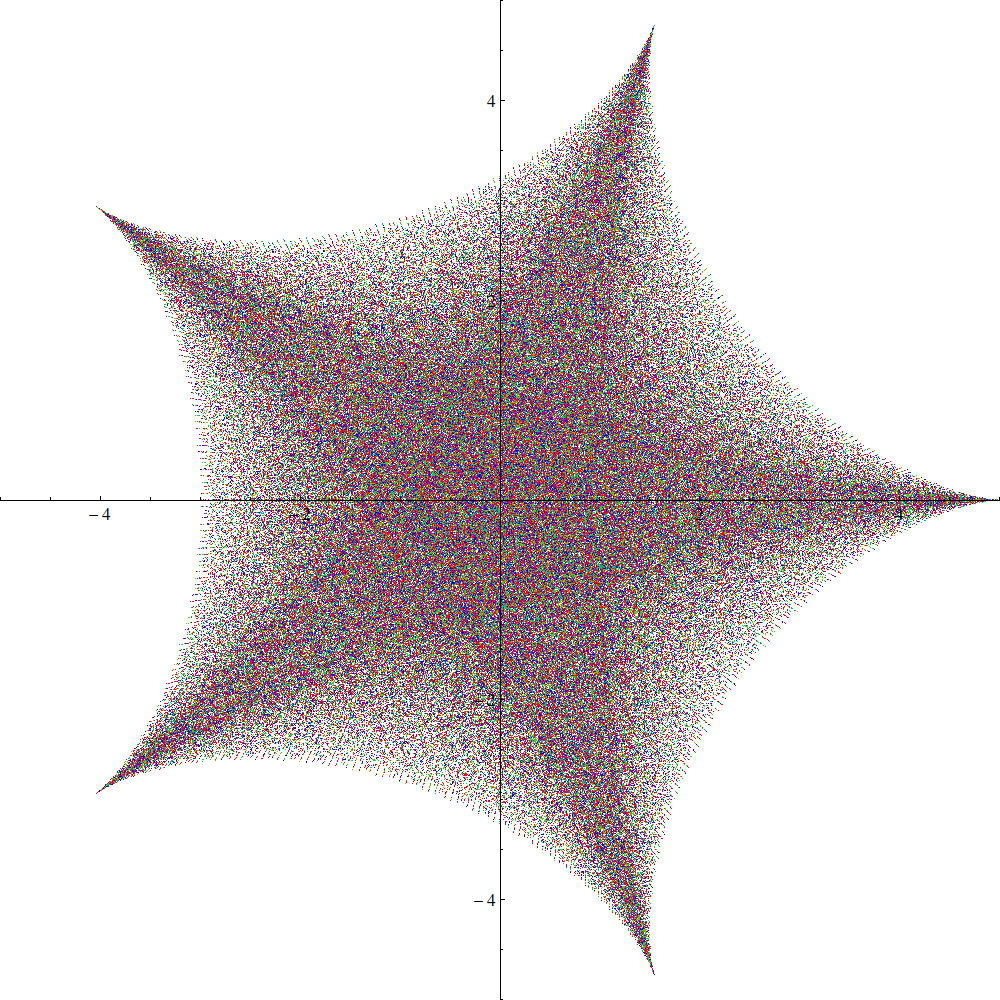}
    		\caption{$p=2221$, $d=5$}
    	\end{subfigure}
    
    	\begin{subfigure}[b]{0.30\textwidth}
    		\includegraphics[width=\textwidth]{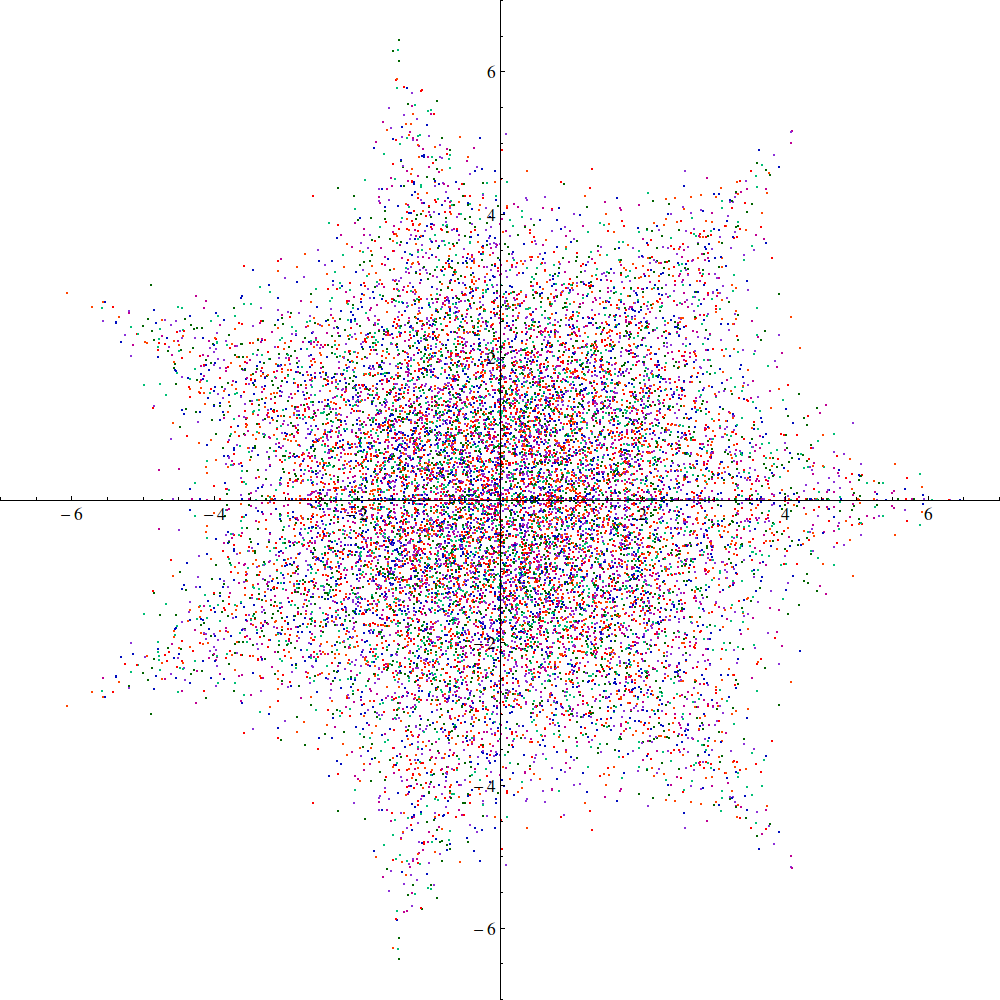}
    		\caption{$p=491$, $d=7$}
    	\end{subfigure}
	\quad
    	\begin{subfigure}[b]{0.30\textwidth}
    		\includegraphics[width=\textwidth]{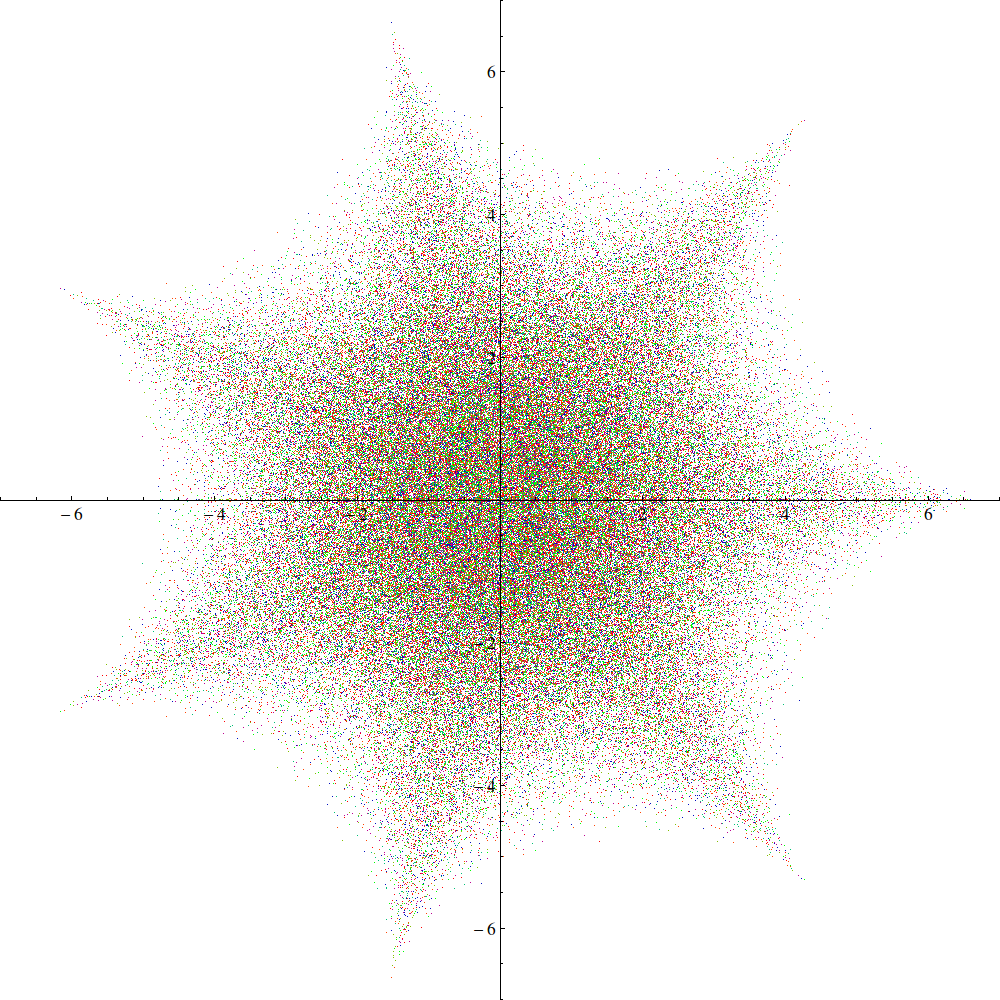}
    		\caption{$p=1597$, $d=7$}
    	\end{subfigure}
	\quad
    	\begin{subfigure}[b]{0.30\textwidth}
    		\includegraphics[width=\textwidth]{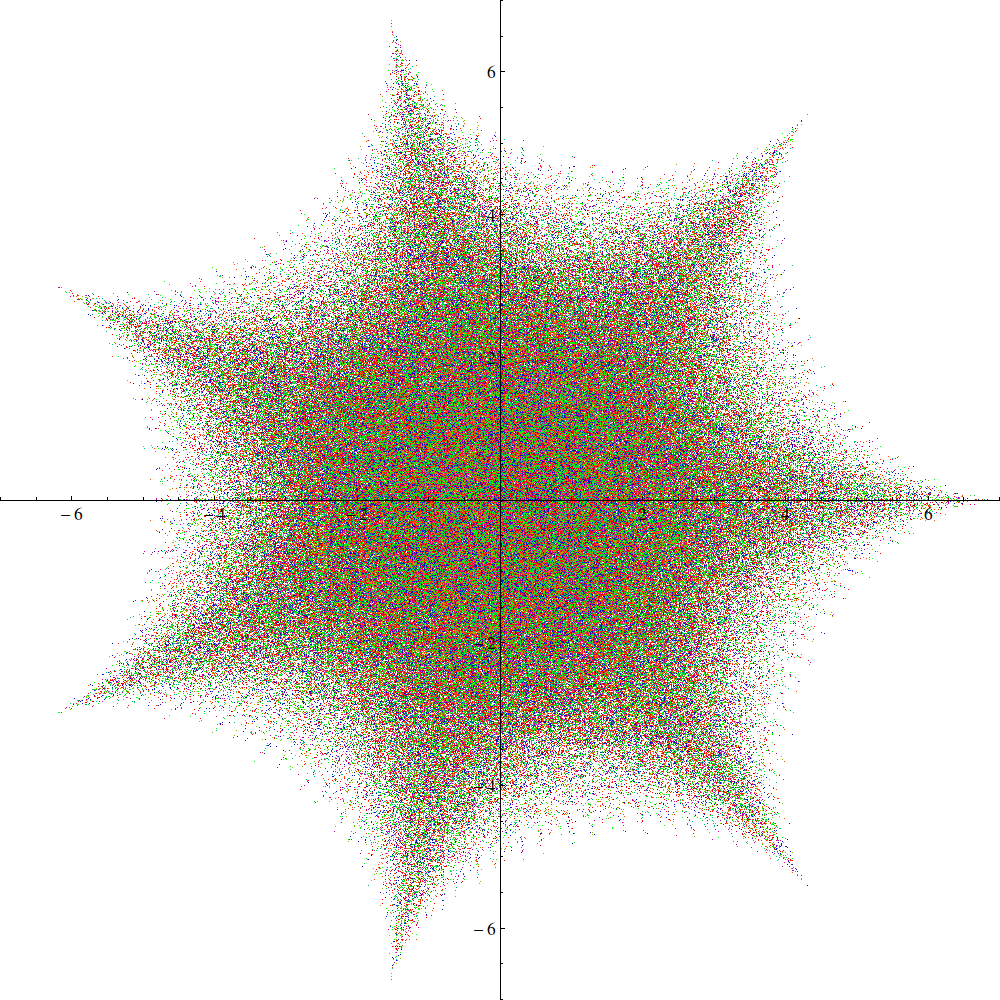}
    		\caption{$p=2969$, $d=7$}
    	\end{subfigure}
    	\caption{For the primes $d = 3,5,7$, the values $K(a,b,p,d)$ with $0\leq a,b \leq p-1$ ``fill out'' the closure $\h_d$ of the bounded
	region determined by the hypocycloid \eqref{eq:Parameter}; see Theorem \ref{thm:hypocycloid}.}
    	\label{fig:hypocycloid}
    \end{figure}

    \begin{lemma}\label{Lemma:UniformDistribution}
        Fix a positive integer $d$.  
        For each odd prime power $q=p^\alpha$ with $p \equiv 1 \pmod{d}$, let $\omega_q$ denote a 
        primitive $d$th root of unity modulo in $\Z/q\Z$.  Let $\omega_q^{-1}$ denote the
        inverse of $\omega_q$ modulo $q$.  For each fixed $b \in \{0,1,\ldots,q-1\}$, the sets
        \begin{equation}\label{eq:Sq}
            S_q=\Big\{\Big(\frac{a+b}{q},\frac{a\omega_q+b\omega_q^{-1}}{q},\ldots,
            \frac{a\omega_q^{\phi(d)-1}+b\omega_q^{-\phi(d)+1}}{q}\Big) : 0\leq a \leq q-1\Big\}
        \end{equation}
        in $\R^{\phi(d)}$ are uniformly distributed modulo $1$ as $q \to \infty$.  
    \end{lemma}

    \begin{proof}
        Fix a positive integer $d$.
        Weyl's criterion asserts that $S_q$ is uniformly distributed modulo 
        $1$ if and only if $$\lim_{q \to \infty} \sum_{\vec{x} \in S_q}\e{\vec{x} \cdot \vec{y}}=0$$ for all nonzero $\vec{y} \in \R^{\phi(d)}$ \cite{weyl}.    
        Fix a nonzero $\vec{y} \in \R^{\phi(d)}$ and let $0 \leq b \leq q-1$. For $\vec{x} \in S_q$, write $\vec{x}= \vec{x}_1+\vec{x}_2$, in which
        \begin{equation*}
            \vec{x}_1=\Big(\frac{a}{q},\frac{a\omega_q}{q},\ldots,\frac{a\omega_q^{\phi(d)-1}}{q}\Big), 
            \qquad \vec{x}_2=\Big(\frac{b}{q},\frac{b\omega_q^{-1}}{q},\ldots,\frac{b\omega_q^{-\phi(d)+1}}{q}\Big);
        \end{equation*}
        Note that $\vec{x}_1$ depends on $a$ whereas $\vec{x}_2$ is fixed since we regard $b$ as constant.
        A result of Myerson \cite[Thm.~12]{myerson} (see also \cite[Lem.~6.2]{duke}) asserts that the sets
        \begin{equation*}
            \Big\{ \frac{a}{q}\big(1,\omega_q, \omega_q^2,\dots, \omega_q^{\phi(d)-1}\big) : 0 \leq a \leq q-1 \Big\}
            \subseteq [0,1)^{\phi(d)}
        \end{equation*}
        are uniformly distributed modulo $1$ as $q=p^\alpha$ tends to infinity; this requires the assumption that $p \equiv 1 \pmod{d}$.  Thus,
        \begin{equation*}
            \frac{1}{|S_q|} \sum_{\vec{x} \in S_q} \e{\vec{x} \cdot \vec{y}} 
            = \frac{1}{|S_q|} \sum_{a=0}^{q-1} e\big((\vec{x}_1+\vec{x}_2) \cdot \vec{y} \big) 
            = e(\vec{x}_2 \cdot \vec{y}) \frac{1}{|S_q|}\sum_{a=0}^{q-1} e(\vec{x}_1 \cdot \vec{y})
        \end{equation*}
        tends to zero, so Weyl's criterion ensures that the sets $S_q$ are uniformly distributed modulo $1$ as $q \to \infty$.
    \end{proof}

    The following theorem explains the asymptotic behavior exhibited in Figure \ref{fig:hypocycloid}.

    \begin{theorem}\label{thm:hypocycloid}
        Fix an odd prime $d$.
        \begin{enumerate}\addtolength{\itemsep}{0.5\baselineskip}
            \item[(a)] For each odd prime power $q=p^\alpha$ with $p \equiv 1 \pmod{d}$, the values of $K(a,b,q,d)$ are contained in $\h_d$, the closure of the region
            	bounded by the $d$-cusped hypocycloid centered at $0$ and with a cusp at $d$.
            \item[(b)] Fix $\epsilon > 0$, $b \in \Z$, and let $B_{\epsilon}(w)$ be an open ball of radius $\epsilon$ centered at $w\in \h_d$.
                For every sufficiently large odd prime power $q=p^\alpha$ with $p \equiv 1 \pmod{d}$,
                there exists $a \in \Z/q\Z$ so that $K(a,b,q,d) \in B_{\epsilon}(w)$.
        \end{enumerate}
    \end{theorem}

    \begin{proof}
        (a) Suppose that $q = p^{\alpha}$ is an odd prime power and $p \equiv 1 \pmod d$.
        Let $g$ be a primitive root modulo $q$ and define
        $u = g^{\phi(q)/d}$, so that $u$ has multiplicative order $d$ modulo $q$.  
        Since $p$ and $p-1$ are relatively prime, 
        $p-1 \nmid p^{\alpha-1}(\frac{p-1}{d})=\phi(q)/d$ and hence $u \not \equiv 1 \pmod{p}$.  Thus, $u-1$ is a unit modulo $q$, from which it follows that
        \begin{equation}\label{eq:cyclotomic}
            1+u+ \cdots +u^{d-1} \equiv 0 \pmod{q}. 
        \end{equation}
        
        Let $\T$ denote the unit circle in $\C$ and define $f:\T^{d-1}\to\C$ by
        \begin{equation}\label{eq:hypo}
            f(z_1,z_2,\ldots ,z_{d-1})=z_1+z_2+\cdots +z_{d-1}+\frac{1}{z_1z_2\cdots z_{d-1}};  
        \end{equation}
        it is well-known that the image of this function is the filled hypocycloid defined by \eqref{eq:Parameter}; 
        see \cite{cooper2007almost,gausscyclotomy,kaiser}.
        For $k=1,2,\ldots,d-1$, let
        \begin{equation*}
                \zeta_k=e\Big(\frac{au^{k-1}+bu^{-(k-1)}}{q} \Big)
        \end{equation*}
        and use \eqref{eq:cyclotomic} to conclude that
          \begin{align*}
            K(a,b,q,d) 
            &= \sum_{k=0}^{d-2} e \Big(\frac{au^k+bu^{-k}}{q} \Big) + e\Big(\frac{au^{d-1}+bu^{-(d-1)}}{q} \Big) \\
            &= \sum_{k=0}^{d-2} e \Big(\frac{au^k+bu^{-k}}{q} \Big) + e\Bigg(\frac{-a\sum_{k=0}^{d-2}u^k - b\sum_{k=0}^{d-2} u^{-k}}{q} \Bigg) \\
            &= \sum_{k=1}^{d-1}\zeta_k + \frac{1}{\zeta_1\zeta_2 \cdots \zeta_{d-1}}. 
          \end{align*}
        Thus $K(a,b,q,d)$ is contained in $\h_d$.  
        \medskip
        
        \noindent (b) Fix $\epsilon > 0$, $b \in \Z$, and let $B_{\epsilon}(w)$ be an open ball of radius $\epsilon$ centered at $w\in \h_d$.
        Let $f:\T^{d-1}\to\C$ denote the function defined by \eqref{eq:hypo} and let $\vec{z} \in \T^{d-1}$ satisfy $f(\vec{z}) = w$.
        The compactness of $\T^{d-1}$ ensures that $f$ is uniformly continuous, so there exists $\delta > 0$ so that
        $|f(\vec{z}) - f(\vec{x})| < \epsilon$ whenever $\| \vec{x} - \vec{z} \| < \delta$ (here we use the norm induced by the standard
        embedding of the torus $\T^{d-1}$ into $d$-dimensional Euclidean space).    
        Since $d$ is prime, $\phi(d)=d-1$ and hence
        Lemma \ref{Lemma:UniformDistribution} ensures that for each fixed $b$, the sets $S_q$ in $\R^{d-1}$ defined by \eqref{eq:Sq}
        are uniformly distributed mod $1$.  So for $q$ sufficiently large, there exists 
        \begin{equation*}
        \vec{x} = \Big(\frac{a+b}{q},\frac{a\omega_q+b\omega_q^{-1}}{q},\ldots,
            \frac{a\omega_q^{\phi(d)-1}+b\omega_q^{-\phi(d)+1}}{q}\Big) \in S_q,
        \end{equation*}
        so that $\| \vec{x} - \vec{z}\| < \delta$ holds.  Then $K(a,b,q,d) = f(\vec{x})$ belongs to $B_{\epsilon}(\vec{w})$. 
    \end{proof}

\section{Variants of hypocycloids}
    A glance at Figure \ref{fig:various1} suggests that hypocycloids are but one of many shapes
    that the values of generalized Kloosterman sums ``fill out.''  With additional work, a variety of
    results similar to Theorem \ref{thm:hypocycloid} can be obtained.  The following theorem is illustrated in Figure 
    \ref{fig:d=9}.

    \begin{figure}
        	\centering
        	\begin{subfigure}[b]{0.30\textwidth}
        		\includegraphics[width=\textwidth]{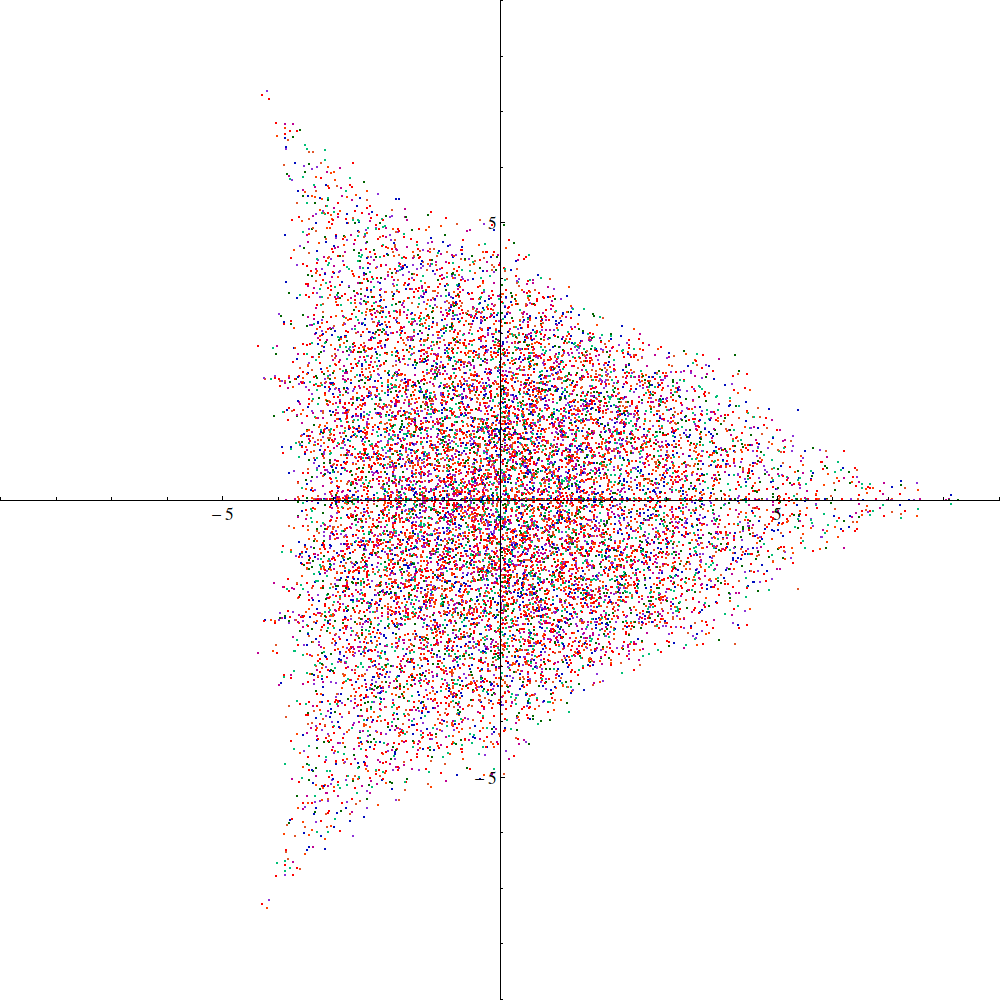}
        		\caption{$p=523$}
        	\end{subfigure}
        	\begin{subfigure}[b]{0.30\textwidth}
        		\includegraphics[width=\textwidth]{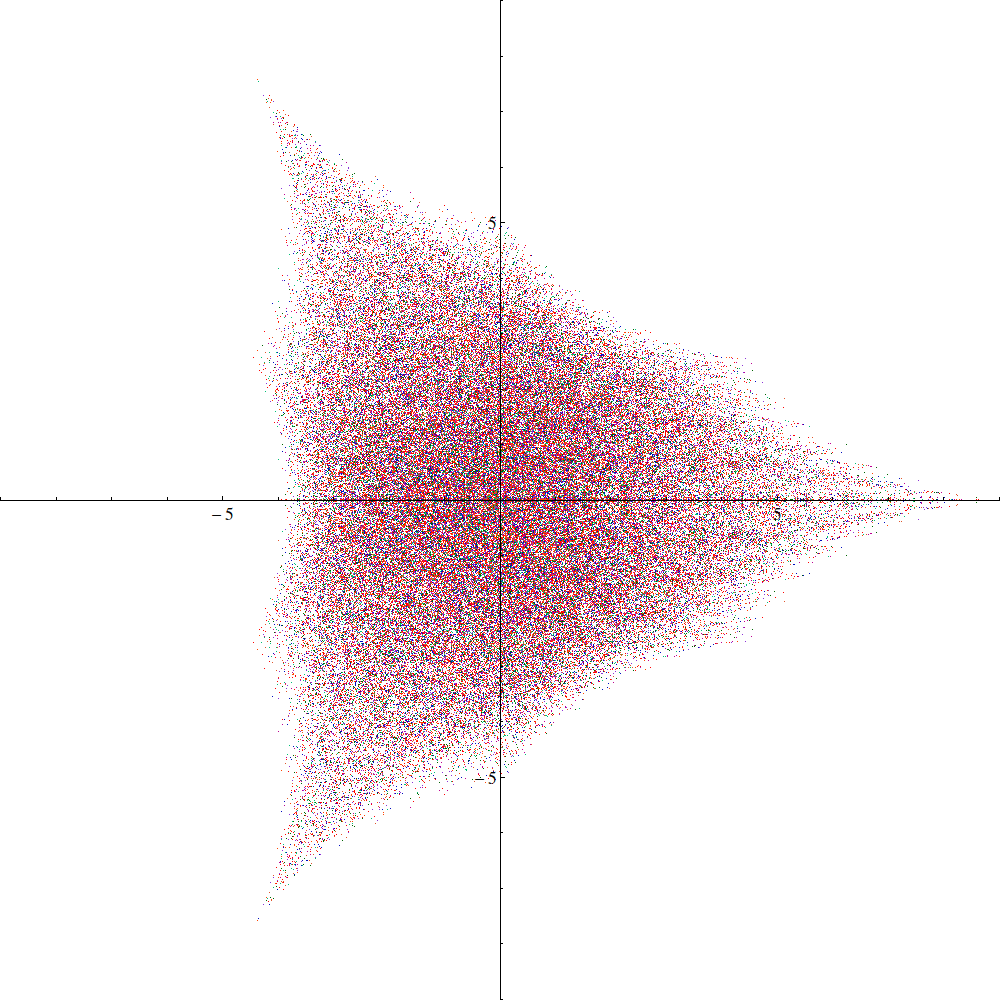}
        		\caption{$p=1621$}
        	\end{subfigure}
        	\begin{subfigure}[b]{0.30\textwidth}
        		\includegraphics[width=\textwidth]{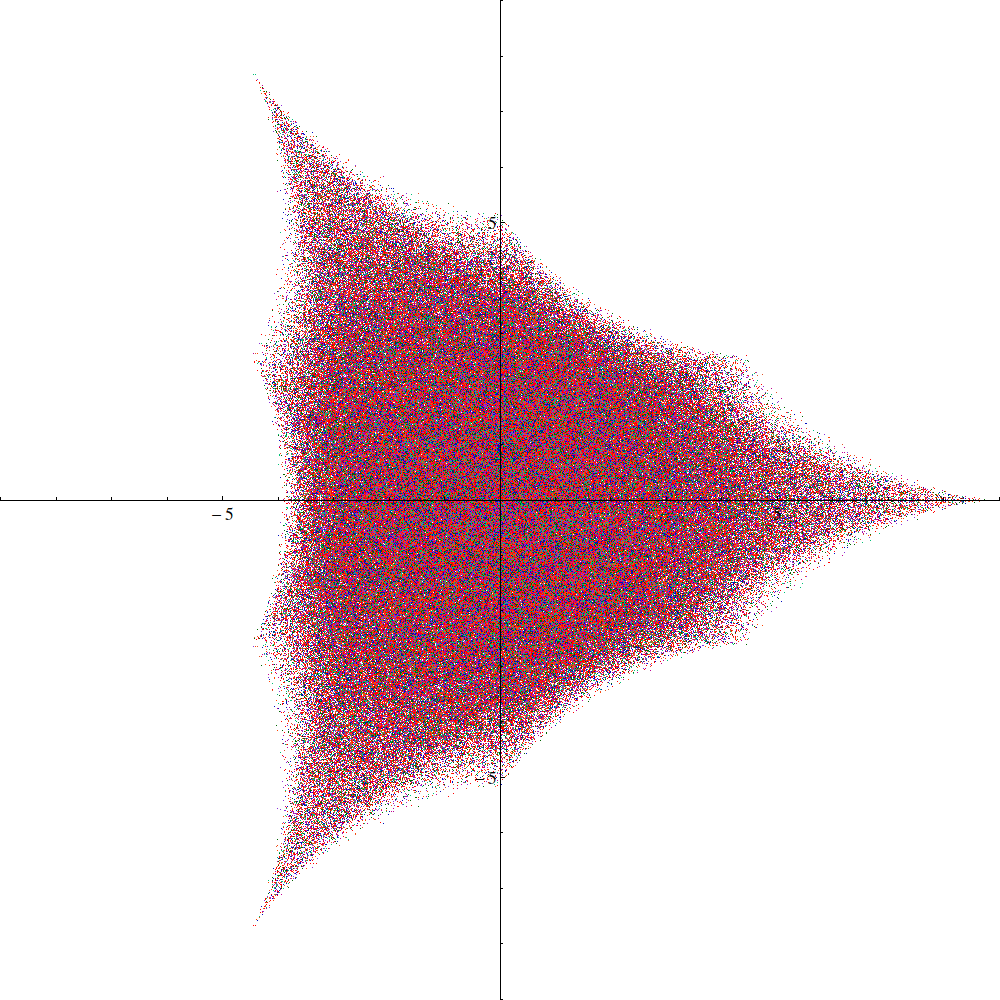}
        		\caption{$p=3511$}
        	\end{subfigure}
    	\caption{The values $\{ K(a,b,p,9) : 0 \leq a,b \leq p-1\}$ for several primes $p\equiv 1\pmod{9}$; see Theorem 
	\ref{theorem:tiled}.}
    	\label{fig:d=9}
    \end{figure}

    \begin{theorem}\label{theorem:tiled}
          Let $p \equiv 1 \pmod{9}$ be an odd prime.
          If $q=p^\alpha$ and $\alpha \geq 1$, then the values $\{K(a,b,q,d) : 0 \leq a,b \leq p-1\}$ are contained in the threefold sum
          \begin{equation}\label{eq:3sum}
                  \h_3 + \h_3 + \h_3 = \{ w_1+ w_2 + w_3 : w_1,w_2,w_3 \in \h_3\}
          \end{equation}
          of the filled deltoid $\h_3$.  Moreover,
          as $q \to \infty$, this shape is ``filled out'' in the sense of Theorem \ref{thm:hypocycloid}.
    \end{theorem}

    \begin{proof}
        Let $g$ be a primitive root modulo $q$ and define $u=g^{\phi(q)/9}$, so that $u$ has multiplicative order $9$ 
        modulo $q$.  Since $p$ and $p-1$ are relatively prime, $p-1 \nmid p^{\alpha-1}(\frac{p-1}{3})=\phi(q)/3$, so $u^3 \not\equiv 1 \pmod{p}$. Thus,
        \begin{equation*}
            u^6 + u^3 + 1 \equiv  (u^9 - 1)(u^3 - 1)^{-1} \equiv 0 \pmod{q},
        \end{equation*}
        so that $u^{6+j} \equiv - u^{3+j} - u^{j} \pmod{q}$ for $j=1,2,3$.  Along similar lines, we have
        $u^{-(6+j)} \equiv -u^{-(3+j)} - u^{-j} \pmod{q}$ for $j=1,2,3$.
        For $k=1,2,\ldots,6$, let
        \begin{equation*}
            \zeta_k = e\Big( \frac{au^k+bu^{-k}}{p} \Big)
        \end{equation*}
        and observe that
        \begin{align*}
            K(a,b,q,9)
            &=\sum_{k=1}^9 \eqsmall{au^k+bu^{-k}} \\
            &= \sum_{k=1}^6 \eqsmall{au^k+bu^{-k}} + \sum_{j=1}^3 \eqsmall{a(-u^{3+j}-u^j)+b(-u^{6+j}-u^{3+j})} \\
            &= \zeta_1 + \zeta_2+\cdots + \zeta_6 + \frac{1}{\zeta_1 \zeta_4} + \frac{1}{\zeta_2 \zeta_5} + \frac{1}{\zeta_3 \zeta_6} \\
            &= \Big( \zeta_1 + \zeta_4 + \frac{1}{\zeta_1 \zeta_4} \Big) +
            \Big( \zeta_2 + \zeta_5 + \frac{1}{\zeta_2 \zeta_5} \Big) +
            \Big( \zeta_3 + \zeta_6 + \frac{1}{\zeta_3 \zeta_6} \Big).
        \end{align*}
        Thus, $K(a,b,q,9)$ belongs to $\h_3 + \h_3 + \h_3$.  
        Since $\phi(9)=6$, Lemma \ref{Lemma:UniformDistribution} ensures that the sets
        \begin{equation*}
            T_q = \left\{\left(\eqsmall{au+bu^{-1}}, \ldots, \eqsmall{au^6+bu^{-6}}\right) : 0\leq a \leq q-1\right\}
        \end{equation*}
        are uniformly distributed modulo $1$ for any fixed $b$.  Thus, the values
        $\{K(a,b,q,d) : 0 \leq a,b \leq p-1\}$ ``fill out'' $\h_3 + \h_3 + \h_3$ as $q \to \infty$.
    \end{proof}

\section{Of squares and Sali\'{e} sums}\label{sec:h=2}
    The preceding subsections concerned the asymptotic behavior of generalized Kloosterman sums
    $K(a,b,p^{\alpha},d)$, in which $d|(p-1)$ and $p^{\alpha}$ tends to infinity.  Here 
    we turn the tables somewhat and consider the sums $K(a,b,p,\frac{p-1}{d})$ for $d = 2^n$.  In general, we take 
    $d$ to be the largest power of two that divides $p-1$; otherwise the cyclic subgroup of $\Z/p\Z$ of order $\frac{p-1}{d}$
    has even order.  This forces $K(a,b,p,\frac{p-1}{d})$ to be real valued, which is uninteresting from our perspective.

    For a fixed odd prime $p$, 
    \begin{equation*}
    T(a,b,p)=\sum_{u=1}^{p-1}\leg{u}{p}\ep{au+bu^{-1}}
    \end{equation*}
    is called a \emph{Sali\'{e} sum};  here $(\frac{u}{p})$ denotes the Legendre symbol.
    Although they bear a close resemblance to classical Kloosterman sums, 
    the values of Sali\'e sums can be explicitly determined \cite{iwaniec2004analytic}.
        If $p$ is an odd prime and $(a,p)=(b,p)=1$, then
        \begin{equation}\label{eq:SalieExplicit}
	T(a,b,p)=
	\begin{cases}
	  2\tau_p\cos\Big(\dfrac{2\pi k}{p} \Big) & \text{if } \leg{a}{p}=\leg{b}{p}=1,\\
	  -2\tau_p\cos\Big(\dfrac{2\pi k}{p} \Big) & \text{if } \leg{a}{p}=\leg{b}{p}=-1, \\
	  0 & \text{otherwise},
	\end{cases}
      \end{equation}
    where $k$ is the square root of $4ab$ and
    \begin{equation*}
        \tau_n=\begin{cases} \sqrt{n} & \mbox{if } n \equiv 1 \pmod{4},\\ i\sqrt{n} & \mbox{if } n \equiv 3 \pmod{4}. \end{cases}
    \end{equation*}

    \begin{figure}
        	\centering
        	\begin{subfigure}[b]{0.30\textwidth}
        		\includegraphics[width=\textwidth]{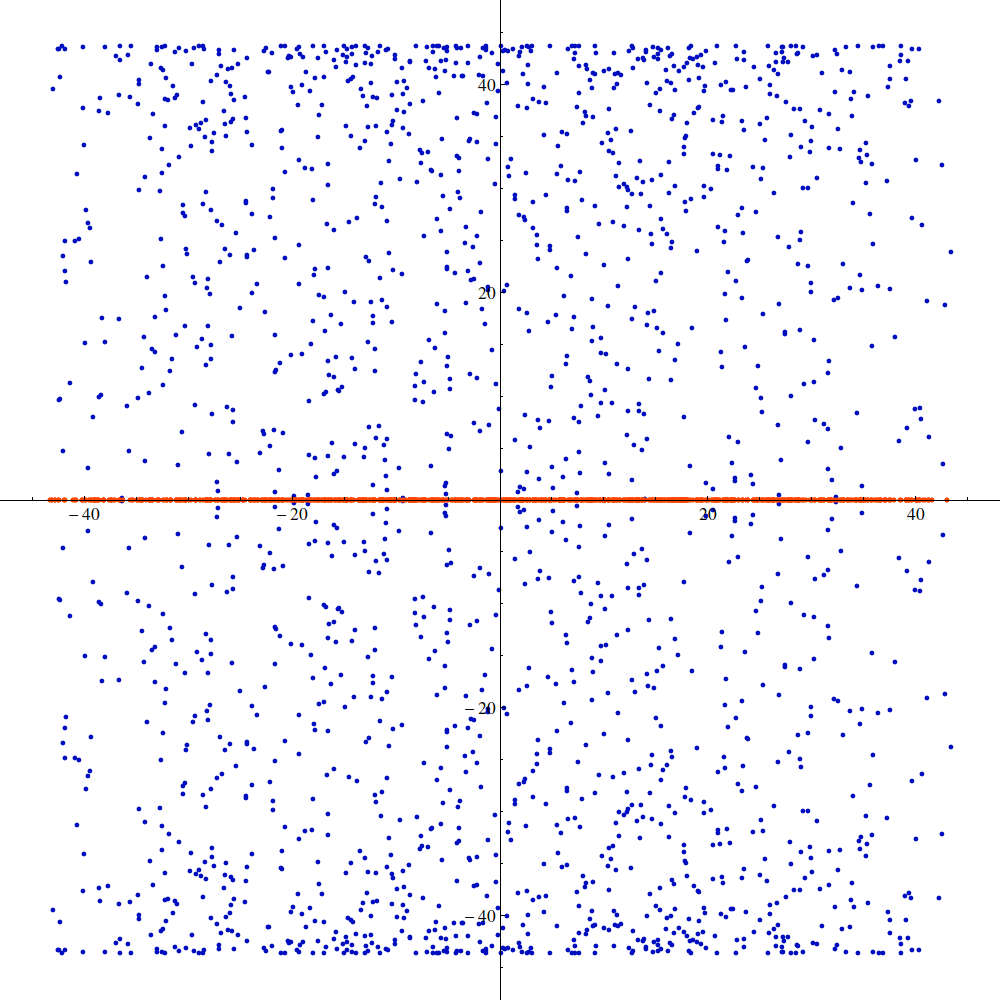}
        		\caption{$p=1907$}
        	\end{subfigure}
        	\begin{subfigure}[b]{0.30\textwidth}
        		\includegraphics[width=\textwidth]{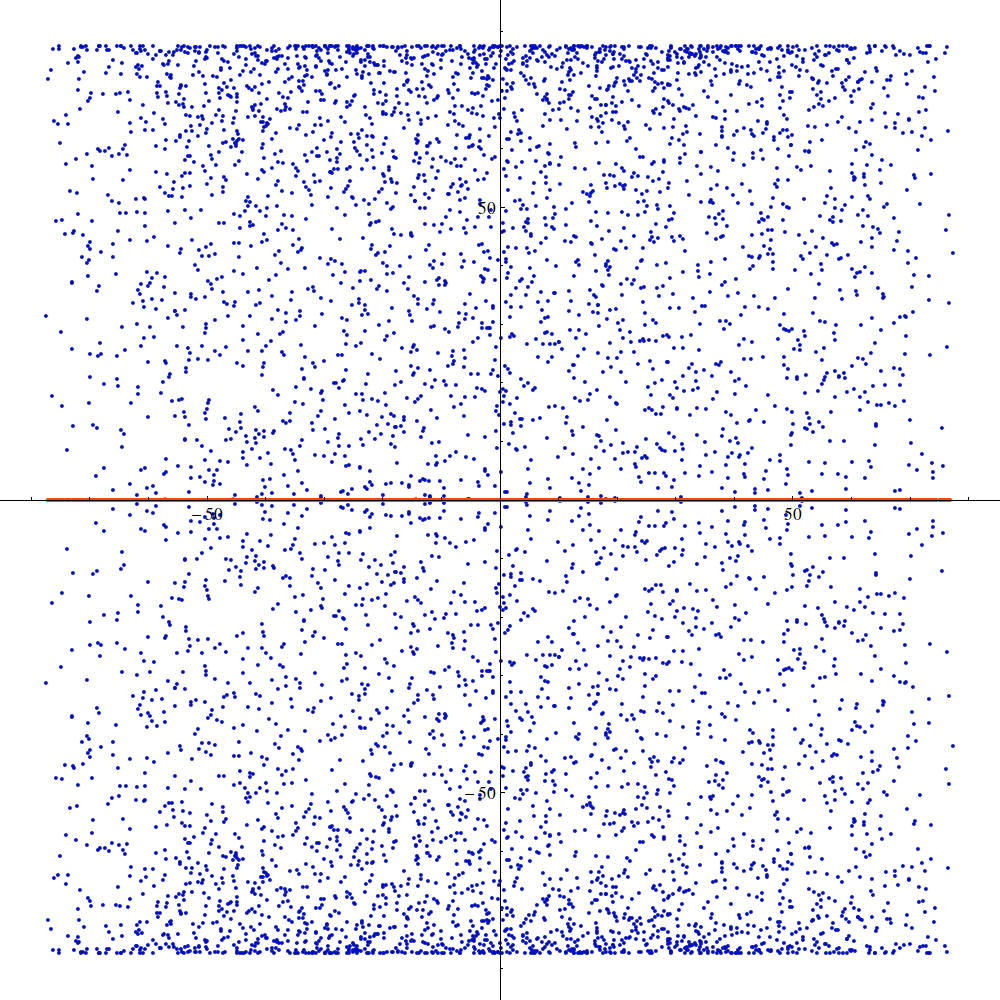}
        		\caption{$p=6007$}
        	\end{subfigure}
        	\begin{subfigure}[b]{0.30\textwidth}
        		\includegraphics[width=\textwidth]{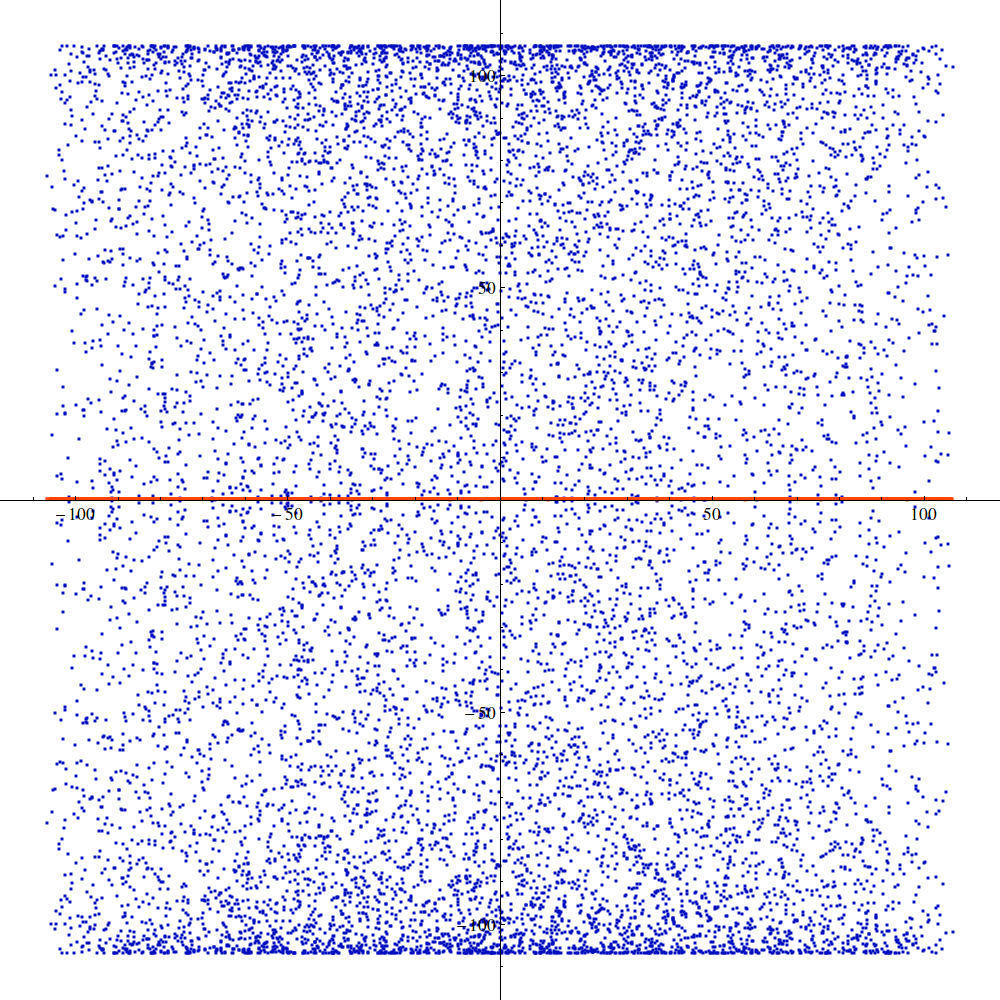}
        		\caption{$p=11447$}
        	\end{subfigure}
    	\caption{Plots of $K(a,b,p,\frac{p-1}{2})$ for primes $p \equiv 3 \pmod{4}$ and $p \nmid ab$.  The images are contained
    	in the square with vertices $(\pm \sqrt{p}, \pm i \sqrt{p})$; see Theorem \ref{h2}.
    	Here $K(a,b,p,\frac{p-1}{2})$ is blue if $(\frac{ab}{p})=1$ and red if $(\frac{ab}{p})=-1$.
	The appearance of the horizontal red line segment is explained by \eqref{eq:h2} and 
	the third condition in \eqref{eq:SalieExplicit}.  The higher density of points near the top and bottom of the blue square
	is due to the cosine term in \eqref{eq:SalieExplicit}.}
    	\label{fig:square}
    \end{figure}

    The following result explains the phenomenon observed in Figure \ref{fig:square}.

    \begin{theorem}\label{h2}
    	Let $p \equiv 3 \pmod{4}$ be an odd prime.  If $p \nmid ab$, then
	\begin{equation}\label{eq:ReIm}
		|\Re K(a,b,p,\tfrac{p-1}{2})| \leq \frac{\sqrt{p}}{2}, \qquad
		|\Im K(a,b,p,\tfrac{p-1}{2})| \leq \frac{\sqrt{p}}{2}.
	\end{equation}
	If $p|ab$, then 
	\begin{equation*}
	  K(a,b,p,\frac{p-1}{2}) = 
	  \begin{cases}
	    \frac{1}{2}\big( ( \frac{b}{p})\tau_p - 1\big) & \text{if $a=0$}, \\
	    \frac{p-1}{2} & \text{if $a=b=0$}.
	  \end{cases}
	\end{equation*}
    \end{theorem}

    \begin{proof}
        Since 
        \begin{align*}
            T(a,b,p) &=\sum_{u^{\frac{p-1}{2}}=1} \epsmall{au+bu^{-1}} - \sum_{u^{\frac{p-1}{2}}=-1}\epsmall{au+bu^{-1}}, \\
                K(a,b,p) &=\sum_{u^{\frac{p-1}{2}}=1} \epsmall{au+bu^{-1}} + \sum_{u^{\frac{p-1}{2}}=-1}\epsmall{au+bu^{-1}},
        \end{align*}
        it follows that 
        \begin{equation}
            K(a,b,p,\tfrac{p-1}{2})=\tfrac{1}{2}\big(T(a,b,p)+K(a,b,p)\big);
            \label{eq:h2}
        \end{equation} 
        we thank Bill Duke for pointing this out to us.
        Since $|T(a,b,p)|\leq 2\sqrt{p}|\cos(\theta)|\leq 2\sqrt{p}$, we have $\frac{1}{2}|T(a,b,p)| \leq \sqrt{p}$.
        On the other hand,  
        the Weil bound for classical Kloosterman sums ensures that $\frac{1}{2}|K(a,b,p)| \leq \sqrt{p}$ if $p \nmid ab$ \cite{weil}.        
        Since $p \equiv 3 \pmod{4}$, the Sali\'{e} sums $T(a,b,p)$ are purely imaginary (or zero), which yields 
        \eqref{eq:ReIm}.  The evaluation of $K(a,b,p,\tfrac{p-1}{2})$ when $p|ab$ is straightforward and omitted.
    \end{proof}

    \begin{figure}
    	\centering
        	\begin{subfigure}[b]{0.45\textwidth}
    		\includegraphics[width=\textwidth]{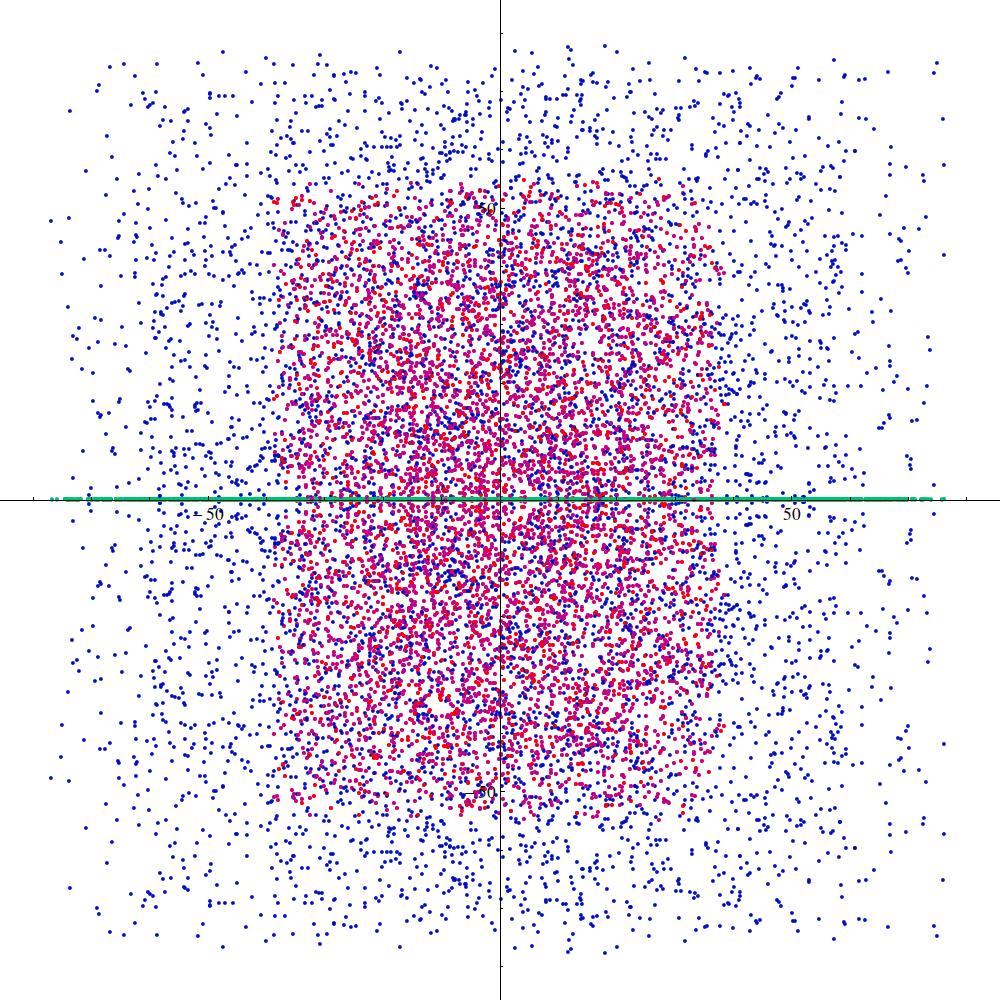}
    		\caption{$p=6053$}
    	\end{subfigure}
    	\quad
        	\begin{subfigure}[b]{0.45\textwidth}
    		\includegraphics[width=\textwidth]{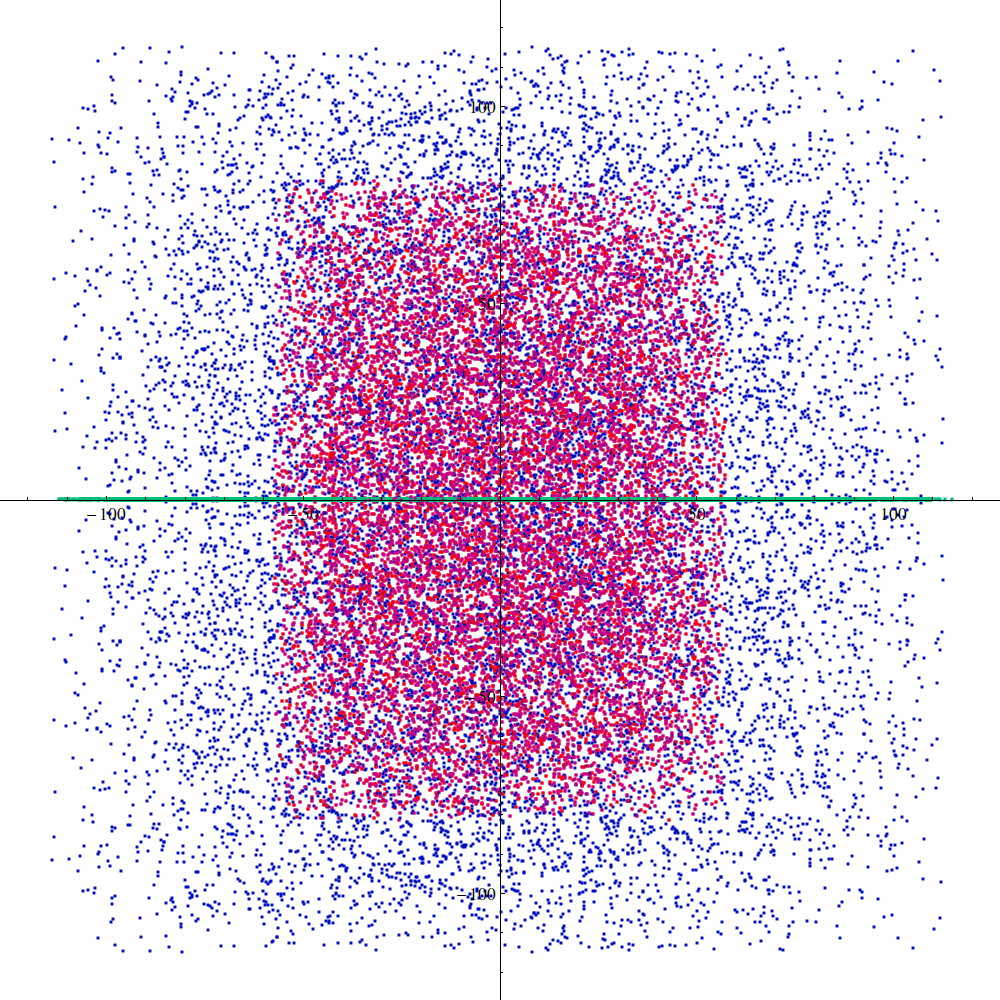}
    		\caption{$p=13309$}
    	\end{subfigure}
    
    	\caption{Generalized Kloosterman sums $K(a,b,p,\frac{p-1}{4})$ with $p\equiv 5\pmod{8}$, $1 \leq a,b \leq p-1$. 
	Write $a=g^r,b=g^s$, where $g$ is a primitive root modulo $p$. $K(a,b,p,\frac{p-1}{4})$ 
	is colored blue, red, green, or purple if $r-s \equiv 0,1,2,$ or $3 \pmod{4}$ respectively.}
    	\label{fig:h=4}
    \end{figure}
    
\section{Boxes in boxes}

    The images for $d=\frac{p-1}{4}$ resemble a rectangle inside a larger square; see Figure \ref{fig:h=4}.  This
    differs significantly from the $d=\frac{p-1}{2}$ case. 
    The following lemma and theorem partially explain the ``box-in-a-box'' behavior of $d=\frac{p-1}{4}$ plots.
    
    \begin{lemma}\label{hreal}
    	Let $p$ be an odd prime of the form $p=2^nd+1$, with $d$ odd and $n \geq 1$. Then $K(a,b,p,\frac{p-1}{2^{n-1}})=2\Re  K(a,b,p,\frac{p-1}{2^n})$. 
    \end{lemma}
    
    \begin{proof}
    	Note that $p=2^nd+1=2^{n-1}(2d)+1$, so $K(a,b,p,\frac{p-1}{2^{n-1}})$ is real-valued. Therefore, it suffices to show that $2K(a,b,p,\frac{p-1}{2^{n}})-K(a,b,p,\frac{p-1}{2^{n-1}})$ is purely imaginary. 
    	\begin{align*}
    	&\overline{2K(a,b,p,\tfrac{p-1}{2^{n}})-K(a,b,p,\tfrac{p-1}{2^{n-1}})} \\
	&\qquad =2\sum_{u^d=1} \epsmall{-au-bu^{-1}} - \sum_{v^{2d}=1}\epsmall{-av-bv^{-1}}	\\
    	&\qquad =\sum_{u^d=1} \epsmall{-au-bu^{-1}} - \sum_{v^d=-1} \epsmall{-av-bv^{-1}} \\
    	&\qquad =\sum_{u^d=(-1)^d} \epsmall{au+bu^{-1}} - \sum_{v^d=-(-1)^d} \epsmall{av+bv^{-1}} \\
    	&\qquad =\sum_{u^{2d}=1} \epsmall{au+bu^{-1}} - 2\sum_{v^d=-(-1)^d} \epsmall{av+bv^{-1}} .
    	\end{align*}
    	Since $d$ is odd, $-(-1)^d=1$, so the above term simplifies to 
	$K(a,b,p,\frac{p-1}{2^{n-1}})-2K(a,b,p,\frac{p-1}{2^{n}})$. Then $2K(a,b,p,\frac{p-1}{2^{n}})-K(a,b,p,\frac{p-1}{2^{n-1}})$ is purely imaginary.
    \end{proof}

    \begin{theorem}\label{thm:h4_bound}
        Let $p \equiv 1 \pmod{4}$ be prime, $1 \leq a,b \leq p-1$. 
        Let $g$ be a primitive root of $(\Z/p\Z)^\times$ and write $a=g^r,b=g^s$. Then 
        \begin{equation*}
            |\Re  K(a,b,p,\tfrac{p-1}{4})| \le \begin{cases} 
            \sqrt{p} &\mbox{if } r-s \equiv 0,2 \pmod{4}, \\[3pt] \frac{\sqrt{p}}{2} &\mbox{if } r - s \equiv 1,3 \pmod{4}. \end{cases}
        \end{equation*}
        Furthermore, if $r-s \equiv 2 \pmod{4}$, then $\Im K(a,b,p,\frac{p-1}{4})=0$.
    \end{theorem}

    \begin{proof}
    By Theorem \ref{h2}, we know that $$K(a,b,p,\tfrac{p-1}{2}) = \frac{T(a,b,p) + K(a,b,p)}{2}.$$ Using Lemma \ref{hreal}, 
    we write $$\Re  K(a,b,p,\tfrac{p-1}{4}) = \frac{T(a,b,p) + K(a,b,p)}{4}.$$ The first half of this fraction is simply a traditional 
    Kloosterman sum, which is real valued and bounded by $\pm 2\sqrt{p}$. Since $p \equiv 1\pmod{4}$, 
    $T(a,b,p)$ is also real valued and bounded by $\pm 2\sqrt{p}$. Note that $4ab$ is a quadratic residue modulo 
    $p$ if and only if $r+s$, and therefore $r-s$, is even. Thus, when $r - s \equiv 0\pmod{4}$ or $r-s \equiv 2 \pmod{4}$, we can say
    \begin{align*}
        |\Re  K(a,b,p,\tfrac{p-1}{4})| &=\frac{K(a,b,p) + T(a,b,p)}{4}\le\frac{2\sqrt{p}+2\sqrt{p}}{4} = \sqrt{p}.
    \end{align*}
    Alternatively, if $r - s \equiv 1\pmod{4}$ or $r-s \equiv 3\pmod{4}$, then $T(a,b,p) = 0$ and
    \begin{align*}
        |\Re  K(a,b,p,\tfrac{p-1}{4})| &=\frac{K(a,b,p)}{4}\le\frac{2\sqrt{p}}{4} = \frac{\sqrt{p}}{2}. 
    \end{align*} 
    Now suppose that $r-s \equiv 2 \pmod{4}$. First, we rewrite $\Im K(a,b,p,\frac{p-1}{4})$ using Lemma \ref{hreal}:
      \begin{align*}
            \Im K(a,b,p,\tfrac{p-1}{4})&=iK(a,b,p,\tfrac{p-1}{4})-i\Re  K(a,b,p,\tfrac{p-1}{4}) \\
            &=iK(a,b,p,\tfrac{p-1}{4})-\frac{i}{2}K(a,b,p,\tfrac{p-1}{2}) \\ 
            &=i\sum_{k=1}^{d}\epsmall{ag^{4k}+bg^{-4k}}-\frac{i}{2}\sum_{k=1}^{2d}\epsmall{ag^{2k}+bg^{-2k}}  \\
            &=\frac{i}{2}\sum_{k=1}^{d}\epsmall{ag^{4k}+bg^{-4k}} - \frac{i}{2}\sum_{k=1}^d\epsmall{ag^{4k+2}+bg^{-4k-2}}  \\
            &=\frac{i}{2}\Big(K(a,b,p,\tfrac{p-1}{4})-K(ag^2,bg^{-2},p,\tfrac{p-1}{4}) \Big). 
      \end{align*}
      Since $r-s \equiv 2 \pmod{4}$, we have $r \equiv 0 \pmod{4}$ and $s \equiv 2 \pmod{4}$, 
      or $r \equiv 1 \pmod{4}$ and $s \equiv 3 \pmod{4}$, up to permutation. Suppose the first case holds. 
      Then we can write $a=g^{4j},b=g^{4k+2}$ for some integers $j,k$. It is easy to check that 
      $K(a,b,p,d)=K(av,bv^{-1},p,d)$ for all $v^d \equiv 1 \pmod{p}$. Using this fact, we obtain
      \begin{align*}
            &K(a,b,p,\tfrac{p-1}{4})-K(ag^2,bg^{-2},p,\tfrac{p-1}{4})\\
            &\qquad =K(g^{4j},g^{4k+2},p,\tfrac{p-1}{4})-K(g^{4j+2},g^{4k},p,\tfrac{p-1}{4})\\
            &\qquad =K(g^{4k},g^{4j+2},p,\tfrac{p-1}{4})-K(g^{4j+2},g^{4k},p,\tfrac{p-1}{4}) = 0.
      \end{align*}
      The second case is similar.
    \end{proof}

    It is apparent from Figure \ref{fig:h=4} that different bounds are obeyed by $K(g^r,g^s,p,\frac{p-1}{4})$ 
    depending on the value of $r-s \pmod{4}$. Theorem \ref{thm:h4_bound} confirms this observation for the real 
    part of the plot.  As seen in Figure \ref{fig:h=4}, the bound of 
    $\Im K(a,b,p,\frac{p-1}{4})$ seems dependent on the value of $r-s$ modulo $p$. We have established the imaginary 
    bound for one value of $r-s$ modulo $p$. We have the following conjecture.
    
    \medskip
    \noindent\textbf{Conjecture}:
    Let $p=4d+1$ be a prime, $g$ a primitive root modulo $p$. Then 
    \begin{equation*}
        |\Im K(g^r,g^s,p,\tfrac{p-1}{4})| \leq \begin{cases} \frac{\sqrt{2p}}{2} 
        &\mbox{if } r-s \equiv 1,3 \pmod{4}, \\ \sqrt{p} &\mbox{if } r-s \equiv 0 \pmod{4}. \end{cases}
    \end{equation*}

\section{Sporadic spiders}\label{sec:bugs}

    \begin{figure}
    	\centering
        	\begin{subfigure}[b]{0.30\textwidth}
    		\includegraphics[width=\textwidth]{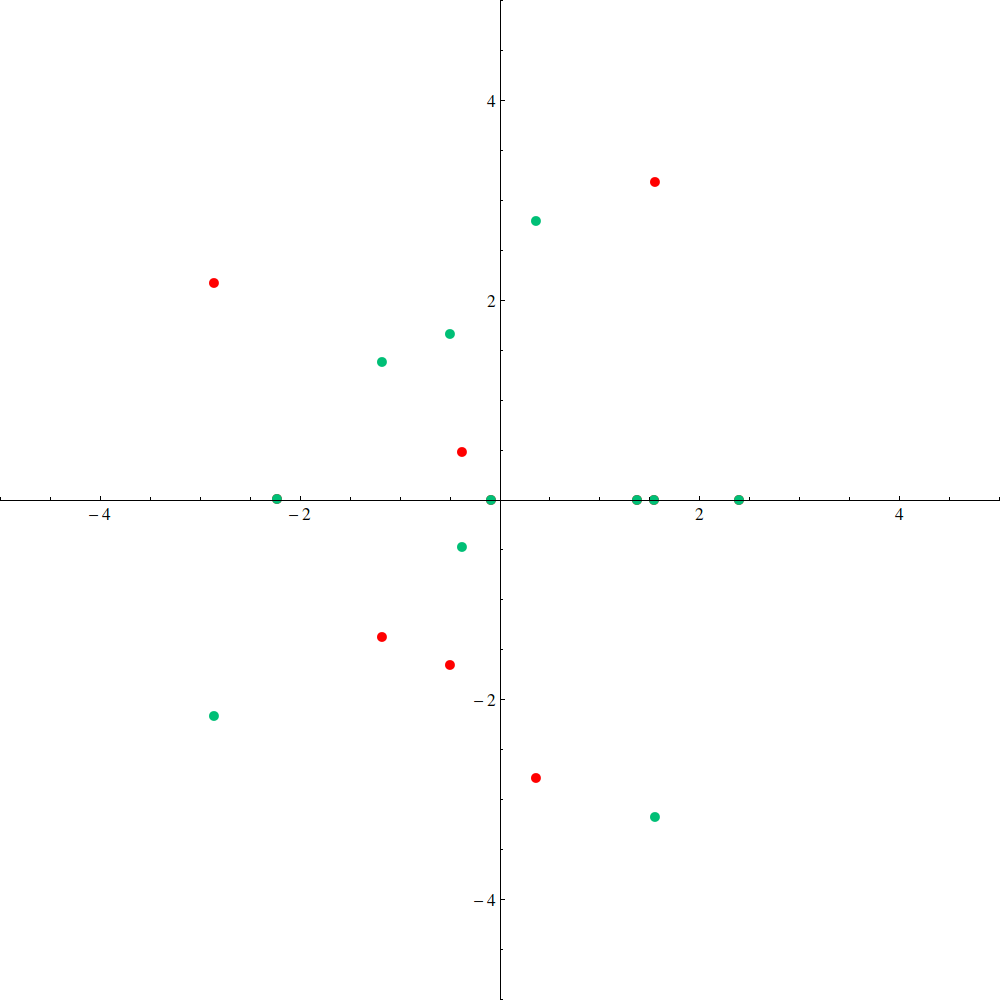}
    		\caption{$m=11$, $d=5$}
    	\end{subfigure}
    	\quad
        	\begin{subfigure}[b]{0.30\textwidth}
    		\includegraphics[width=\textwidth]{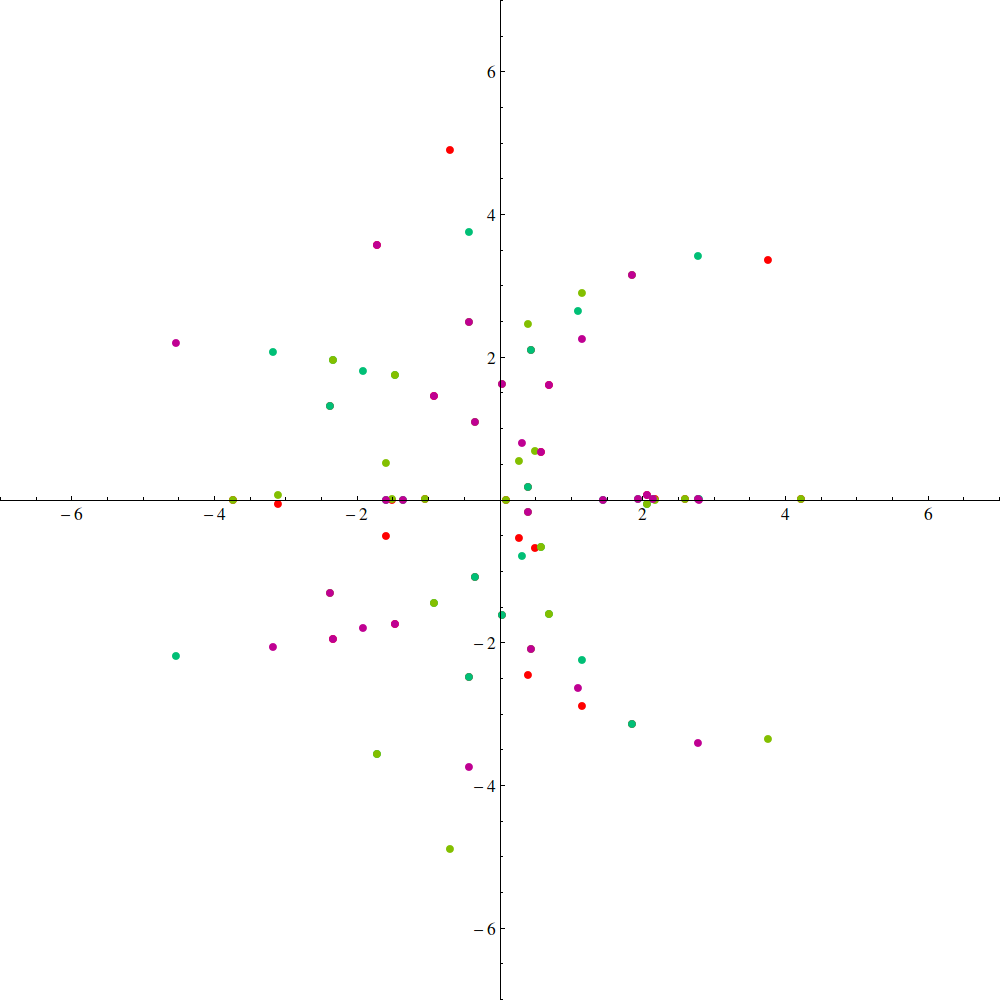}
    		\caption{$m=29$, $d=7$}
    	\end{subfigure}
    	\quad
        	\begin{subfigure}[b]{0.30\textwidth}
    		\includegraphics[width=\textwidth]{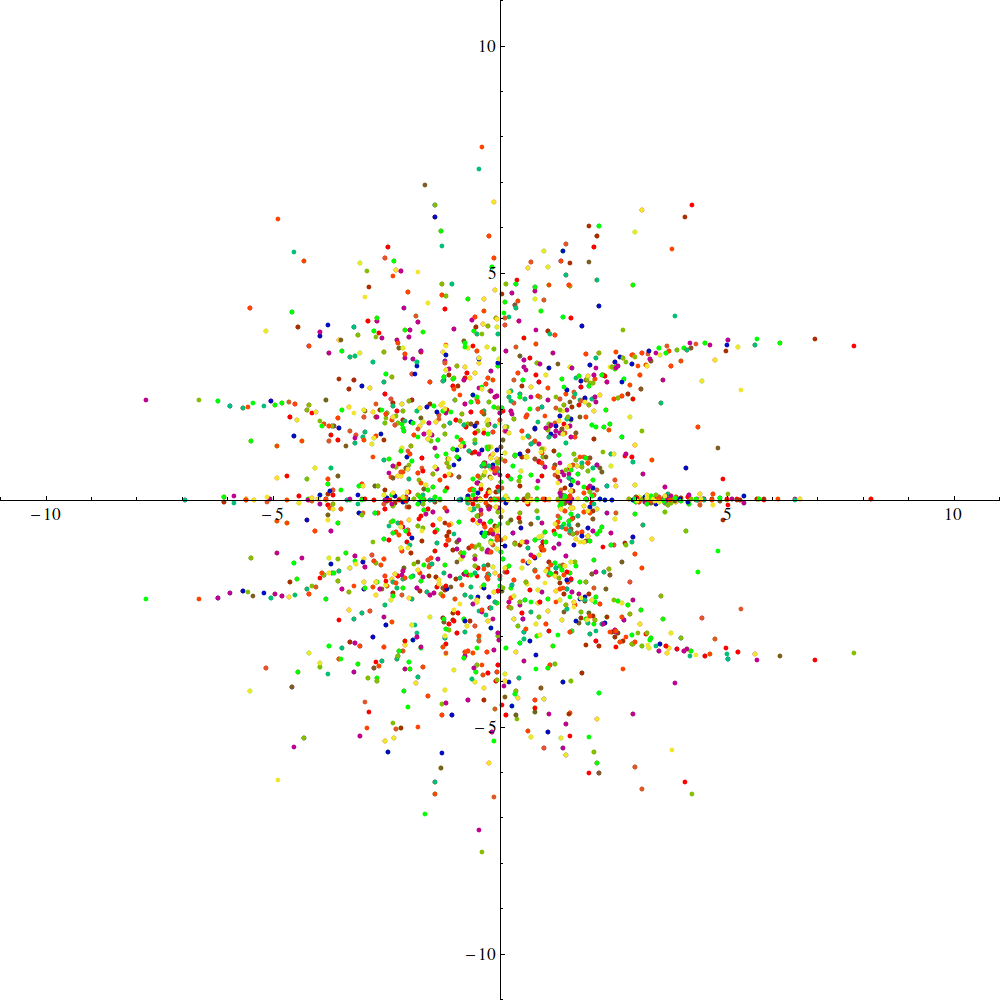}
    		\caption{$m=199$, $d=11$}
    	\end{subfigure}
    
        	\begin{subfigure}[b]{0.30\textwidth}
    		\includegraphics[width=\textwidth]{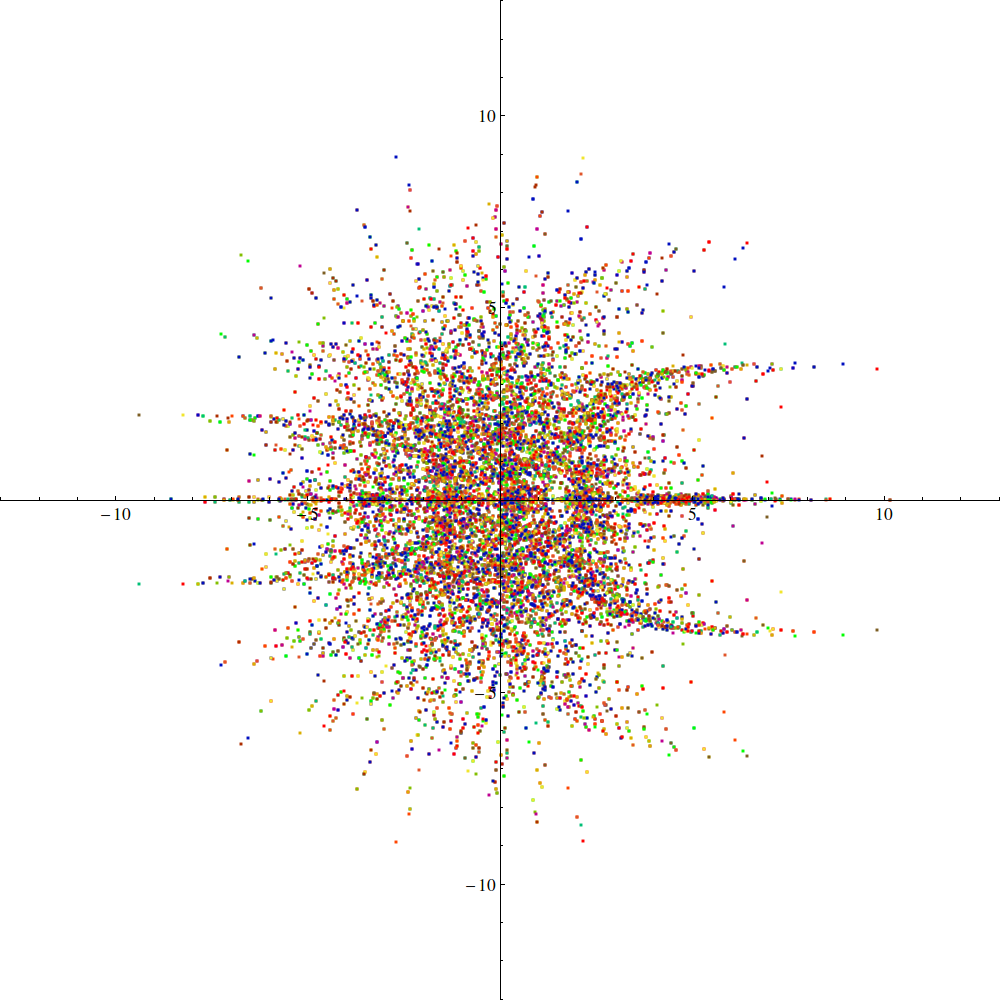}
    		\caption{$m=521$, $d=13$}
    	\end{subfigure}
            \quad
        	\begin{subfigure}[b]{0.30\textwidth}
    		\includegraphics[width=\textwidth]{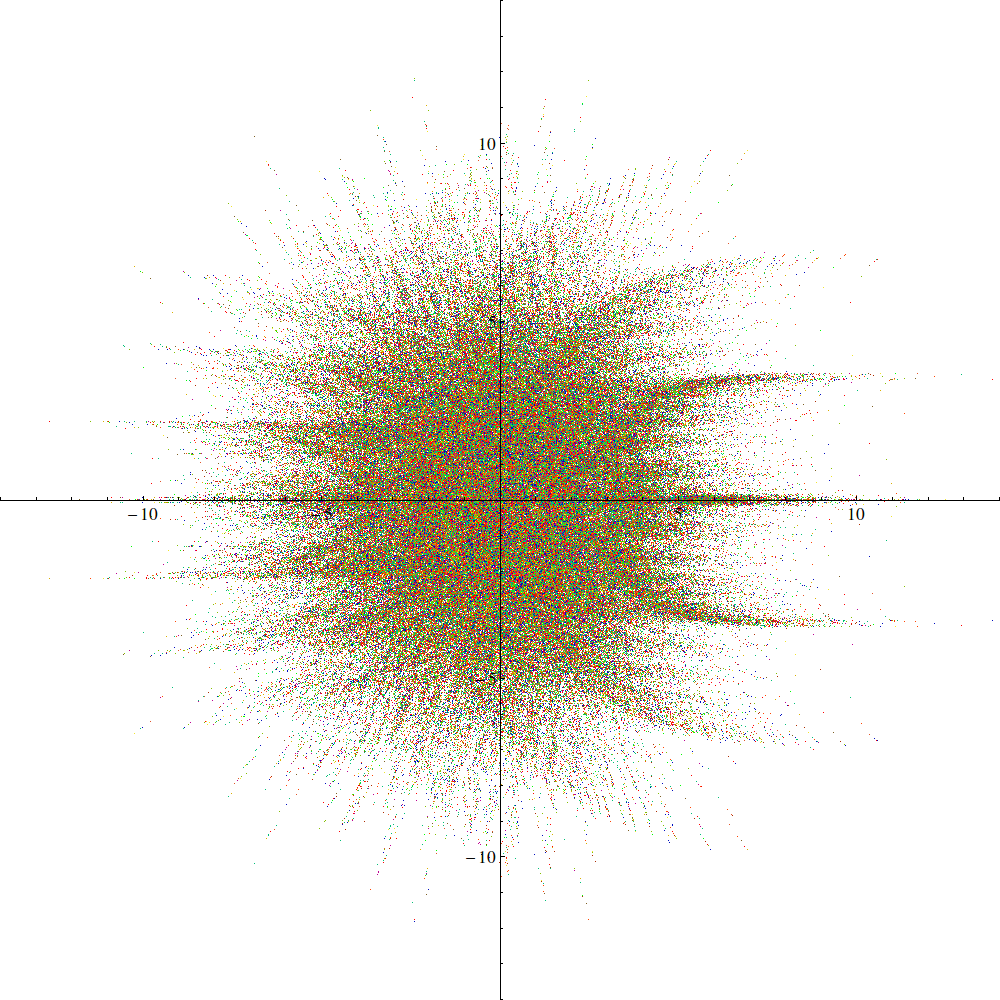}
    		\caption{$m=3571$, $d=17$}
    	\end{subfigure}
    	\quad
        	\begin{subfigure}[b]{0.30\textwidth}
    		\includegraphics[width=\textwidth]{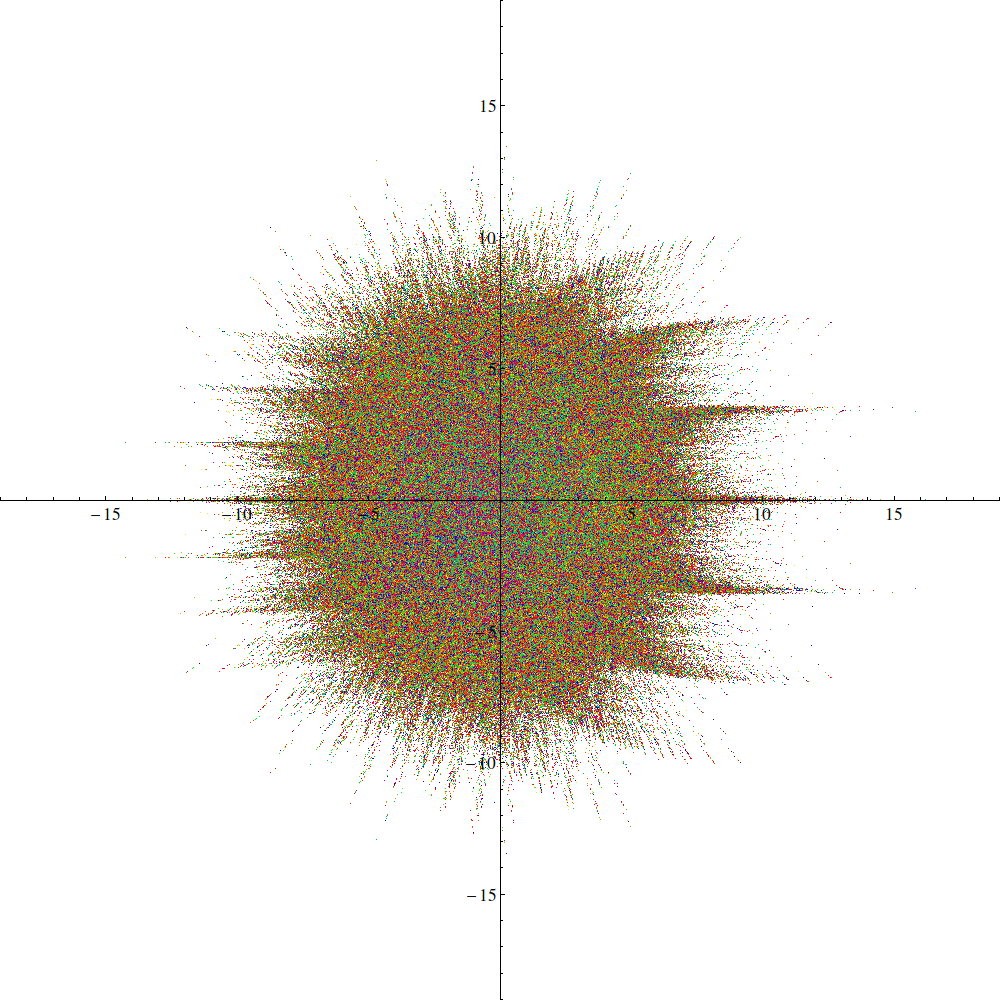}
    		\caption{$m=9349$, $d=19$}
    	\end{subfigure}

    	\caption{Generalized Kloosterman sums $K(a,b,m,\Lambda)$ for some $\Lambda \leq (\Z/m\Z)^\times$ of order $d$.
    	These images resemble fat spiders with a horrifically increasing number of legs.}
    	\label{fig:bugs}
    \end{figure}

    We conclude this note with an investigation of a peculiar and intriguing phenomenon.
    Numerical evidence suggests that for a fixed odd prime $d$, the spider-like image depicted in Figure \ref{fig:bugs}
    appears abruptly for only one specific modulus.
    Figure \ref{fig:fleeting} illustrates the swift coming and going of the ephemeral spider.
    
    \begin{figure}
      \centering
      \begin{subfigure}[b]{0.30\textwidth}
    	\includegraphics[width=\textwidth]{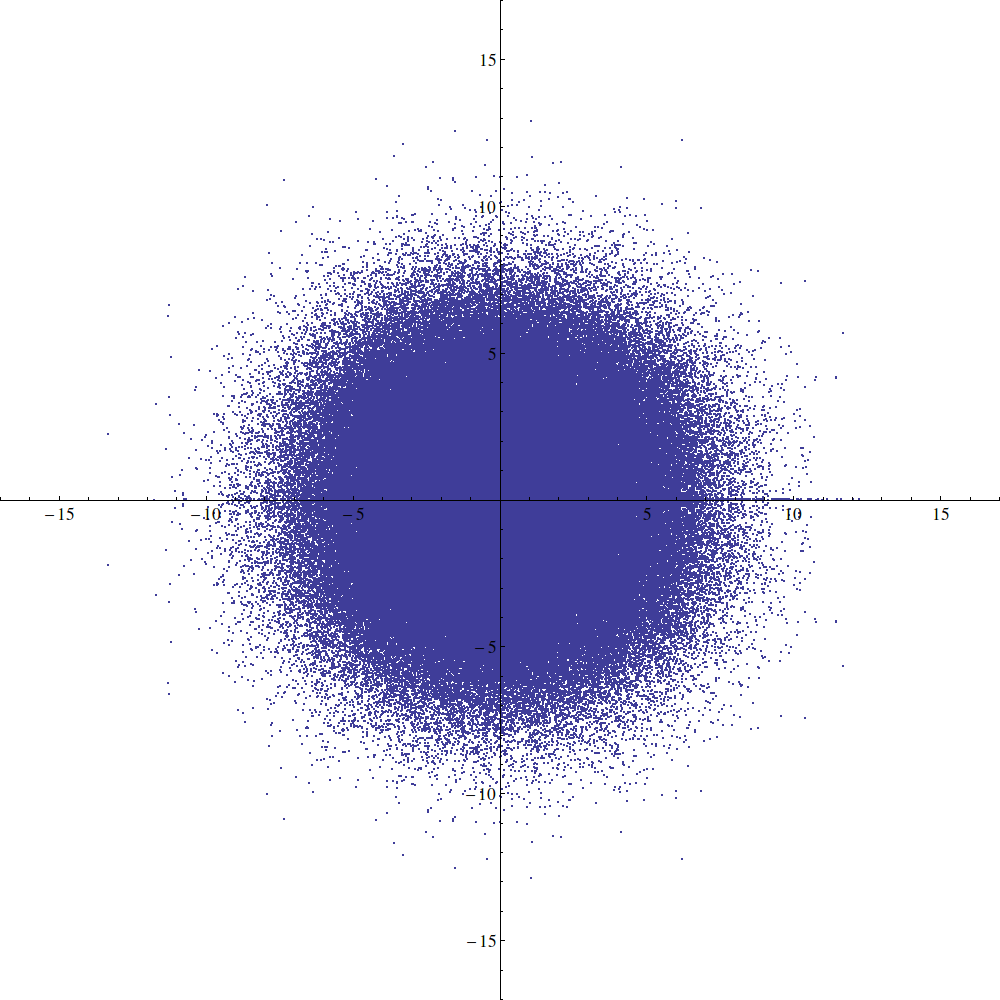}
    	\caption{$p=3469$, $d=17$}
      \end{subfigure}
      \begin{subfigure}[b]{0.30\textwidth}
    	\includegraphics[width=\textwidth]{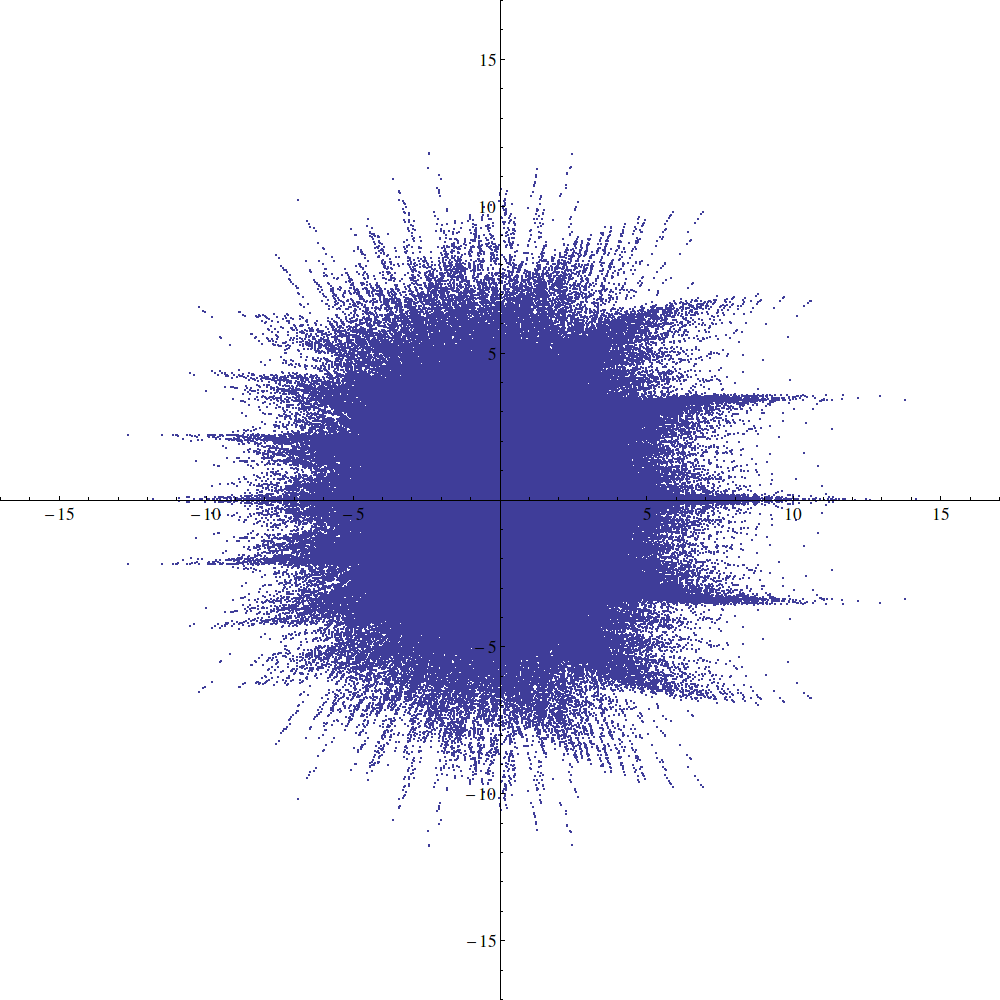}
    	\caption{$p=3571$, $d=17$}
      \end{subfigure}
      \begin{subfigure}[b]{0.30\textwidth}
    	\includegraphics[width=\textwidth]{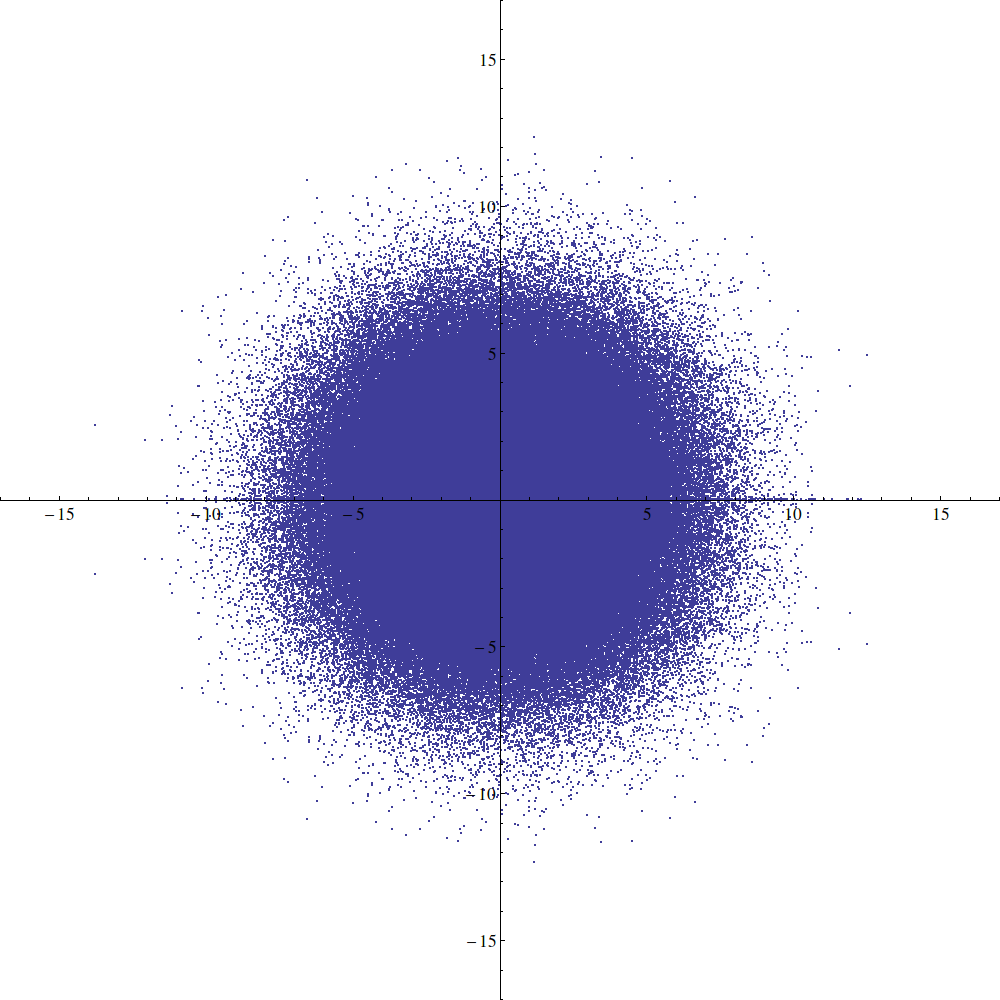}
    	\caption{$p=3673$, $d=17$}
      \end{subfigure}
    
      \caption{$3469,3571,3673$ are three consecutive primes congruent to $1 \pmod{17}$. 
      Theorem \ref{thm:hypocycloid} tells us that the values of the corresponding 
      generalized Kloosterman sums are on their way to filling out $\h_{17}$. 
      The appearance of the spider at $p=3571$ is as surprising as it is fleeting.}
      \label{fig:fleeting}
    \end{figure}

    The moduli that generate the spiders are all of the form $L_{p(n)}$, 
    where $p(n)$ is the $n$th prime and $L_k$ is the $k$th Lucas number; see Table \ref{table:bugs}.
    Recall that the Lucas numbers (sequence 
    \href{https://oeis.org/A000032}{\texttt{A000032}} in the OEIS) are defined by the initial conditions $L_0=2$, $L_1=1$,
     and the recurrence relation $L_n=L_{n-1}+L_{n-2}$ for $n>1$.        
    It is not immediately clear that our pattern can continue indefinitely since for prime $d$ we require
    $(\Z/L_d\Z)^\times$ to have a subgroup of order $d$.  This is addressed by Theorem \ref{Theorem:Lucas}
    below.  We first need the following lemma.

    \begin{table}
        \begin{center}
        \begin{tabular}{|c|c|c|c|c|c|c|c|c|c|}
        	\hline
          $n$ & 3 & 4 & 5 & 6 & 7 & 8 & 9 & 10 & 11\\
            \hline\hline
          $p(n)$ & 5 & 7 & 11 & 13 & 17 & 19 & 23 & 29 & 31 \\
          \hline
          $L_{p(n)}$ & 11 & 29 & 199 & 521 & 3571 & 9349 & 64079 & 1149851 & 3010349\\
          \hline
          $\phi(L_{p(n)})$ & 10 & 28 & 198 & 520 & 3570 & 9348 & 63480 & 1130304 & 3010348 \\
          \hline
        \end{tabular}
        \medskip
        
        \caption{This sequence $L_{p(n)}$, in which $p(n)$ is the $n$th prime and $L_k$ is the $k$th Lucas number;
       see \href{https://oeis.org/A180363}{A180363} in the OEIS.  Although the initial terms in this sequence are prime,
        they are not all so.  Theorem \ref{Theorem:Lucas} ensures that $p(n)$ divides $L_{p(n)}$ for $n \geq 3$.}
        \label{table:bugs}
        \end{center} 
    \end{table}

    \begin{lemma}\label{qlem}
        If $p \geq 5$ is an odd prime, then there is an odd prime $q$ so that $q |L_p$.
    \end{lemma}
    
    \begin{proof}
        If we observe $L_0,L_1,\ldots,L_{11}$ modulo $8$, we get $2,1,3,4,7,3,2,5,7,4,3,7,2,1$.
        Because the first two digits of this sequence are the same as the last and $L_n = L_{n-1} + L_{n-2}$, the sequence repeats. 
        Thus, $L_n$ is never divisible by $8$, and furthermore $L_n > 8$ for all $n \geq 5$. 
        Any integer greater than $8$ and not divisible by $8$ cannot be a power of two. Thus, there exists an odd prime $q$ such that $q|L_p$. 
    \end{proof}

    \begin{theorem}\label{Theorem:Lucas}
        If $p\geq 5$ is an odd prime, then $p |  \phi(L_p)$.
    \end{theorem}
    
    \begin{proof}
        Let $F_n$ denote the $n$th Fibonacci number and let $z(n)$ denote the order of appearance of $n$ \cite[p.~89]{vajda1989fibonacci}. 
        By Lemma \ref{qlem} there is an odd prime $q$ such that $q| L_p$. 
        Using the fact that $F_{2p} = L_pF_p$ \cite[p.~25]{vajda1989fibonacci}, we know that $q| F_{2p}$. 
        Furthermore, $q| F_{z(q)}$ by \cite[p.~89]{vajda1989fibonacci}.  Consequently, 
        $$q| \gcd(F_{2p},F_{z(q)}) = F_{\gcd(2p,z(q))},$$ 
        where we have used that $\gcd(F_a,F_b) = F_{\gcd(a,b)}$ for all $a,b \in \Z^+$ \cite[Theorem 16.3]{koshy2011fibonacci}. 
        Now, set $d = \gcd(2p,z(q))$, and observe that since $p$ is prime, $d = 1$, $2$, $p$ or $2p$. 
        However, $q$ is an odd prime and $q| F_d$. If $d = 1$ or $2$, this implies $q| 1$ because $F_1 = F_2 = 1$, 
        which is impossible because $q$ is an odd prime. Thus, $d = p$ or $2p$. 
        Now, consider the case $d = p$, implying that $q| F_n$. However,  by \cite[p.~29]{vajda1989fibonacci} we know 
        $$L_p^2 - 5F_p^2 = 4(-1)^p,$$ thus implying that $q| 4$ which is impossible because $q$ is an odd prime. 
        Thus, $d = \gcd(2p,z(q)) = 2p$ and therefore $2p| z(q)$. Furthermore, we know $z(q)| q - (\frac{q}{5})$.
        Now, $q$ cannot be 5, because the Lucas numbers are always coprime to 5 \cite[p.~89]{vajda1989fibonacci}. 
        We would like to show that $(\frac{q}{5}) = 1$, because then $$p \big| 2p \big| z(q) \big| (q-1) = \phi(q)\big| \phi(L_p).$$
         
        Thus we must show that $q$ is a quadratic residue modulo $5$. For this, we must again use the fact that $L_p^2 - 5F_p^2 = 4(-1)^p$. 
        Reducing this modulo $q$, we get that $$-5F_p^2 \equiv -4 \pmod{q}.$$ Thus $(\frac{5}{q}) = 1$, 
        and furthermore $(\frac{q}{5})=1$ by quadratic reciprocity. 
    \end{proof}

    For $n \geq 3$, Theorem \ref{Theorem:Lucas} guarantees the existence of an order $p(n)$ subgroup $\Lambda$ of 
    $(\Z/L(p(n))\Z)^\times$.  In principle, this permits the patterns hinted at in 
    Figure \ref{fig:fleeting} continue indefinitely.  However, 
    many questions remain. How can the spider phenomena be formalized? 
    One can immediately recognize a spider when one sees it, but it is more difficult to express the irregularity in a mathematical manner. 
    Further, how does the structure of Lucas numbers for prime indices influence the spider-like images? 
    These are questions we hope to return to at a later time.

\label{Bibliography}

\bibliographystyle{amsplain} 

\bibliography{GKS}

\end{document}